\documentclass[11pt,reqno]{amsproc}

\usepackage[utf8]{inputenc}
\usepackage[T1]{fontenc}
\usepackage[english]{babel}

\usepackage{lineno}
\modulolinenumbers[5]

\usepackage{geometry}
\usepackage{amsmath}
\usepackage{amssymb}
\usepackage{amsfonts}
\usepackage{mathrsfs}
\usepackage{amsthm}
\usepackage{mathtools}
\usepackage{cases}
\usepackage{bm}
\usepackage{bbm}

\usepackage[table,xcdraw,dvipsnames]{xcolor}
\usepackage{graphicx}
\setkeys{Gin}{draft=false}
\graphicspath{{figure/}}

\usepackage{float}
\usepackage{placeins}

\usepackage{subfig}

\usepackage{booktabs}
\usepackage{multirow}
\usepackage{threeparttable}
\usepackage{tabularx}
\usepackage{array}
\usepackage{longtable}
\usepackage{siunitx}

\usepackage{enumerate}
\usepackage[inline]{enumitem}
\usepackage{cite}
\usepackage{datetime}

\usepackage{tikz}
\usetikzlibrary{shapes.geometric, arrows, calc, positioning}

\usepackage{listings}
\usepackage{tcolorbox}
\tcbuselibrary{skins, breakable}

\usepackage{algorithm}
\usepackage{algpseudocode}

\usepackage[colorlinks=true,allcolors=blue]{hyperref}

\newtheorem{definition}{Definition}[section]
\newtheorem{lemma}[definition]{Lemma}
\newtheorem{theorem}[definition]{Theorem}
\newtheorem{proposition}[definition]{Proposition}
\newtheorem{remark}[definition]{Remark}

\newtheorem{example}{Example}[section]

\numberwithin{equation}{section}
\numberwithin{figure}{section}
\allowdisplaybreaks[4]

\tikzset{
  graph node/.style={
    circle, draw=black!80, fill=white,
    inner sep=0pt, minimum size=5pt,
    line width=0.5pt
  },
  graph edge/.style={
    draw=gray!40, thin
  },
  mst edge/.style={
    draw=red!80!black, very thick
  }
}

\begin{document}

\title[Optimal Transport On Graphs]{
Newton's Method for Optimal Transport Problem on Graphs}

\author{Qujiangxue Chen}
\address{Department of Applied Mathematics, The Hong Kong Polytechnic University, Kowloon, Hong Kong SAR, P.R. China}
\email{23036943R@connect.polyu.hk}

\author{Jianbo Cui}
\address{Department of Applied Mathematics, The Hong Kong Polytechnic University, Kowloon, Hong Kong SAR, P.R. China}
\email{jianbo.cui@polyu.edu.hk}

\author{Dieci Luca}
\address{School of Mathematics, Georgia Institute of Technology, Atlanta, GA 30332 USA}
\email{luca.dieci@math.gatech.edu}

\author{Haomin Zhou}
\address{School of Mathematics, Georgia Institute of Technology, Atlanta, GA 30332 USA}
\email{hmzhou@math.gatech.edu}

\thanks{The research is partially supported by NSFC grant 12522119, NSFC grant 12301526, MOST National Key \& Program No. 2024YFA1015900, the Hong Kong Research Grant Council GRF grant 15302823, NSFC/RGC Joint Research Scheme N\_PolyU5141/24, the internal grants (P0045336, P0046811) from the Hong Kong Polytechnic University, and the CAS AMSS-PolyU Joint Laboratory of Applied Mathematics. Zhou is partially supported by NSF grants DMS-2307465 and DMS-2510829.}


\begin{abstract}
In this paper, we study dynamical optimal transport on a connected graph from the perspective of the Benamou–Brenier formulation, where densities are assigned to vertices and velocities to edges. 
But, directly using Newton’s method on the resulting nonlinear systems encounters two potential difficulties: (i) If the graph contains cycles, edge variables are not unique, and (ii) there is no guarantee that the density variables remain positive. To address these challenges, we introduce a finite difference-type  Newton method that eliminates cycle-induced redundancies through a spanning-tree gauge, resulting in a reduced set of independent variables and a well-posed, sparse linear system. For the lattice graph arising from the continuous OT problem, density positivity can also be  guaranteed by using an upwind discretization subject to a CFL-type condition. We further demonstrate the versatility of the proposed scheme by applying it to a range of problems, including optimal transport on lattices and random graphs, inverse optimal transport problems, and social network analysis.
\end{abstract}

\maketitle
\section{Introduction}
Optimal transport (OT) provides a fundamental geometric framework for comparing probability distributions and has become a central tool in a wide range of scientific disciplines, including image processing, machine learning, and network analysis \cite{peyre2019computational, villani2009optimal}. 
Originating from the mass transfer problem proposed by Monge \cite{Monge1781Deblais} in 1781 and later relaxed by Kantorovich \cite{kantorovich1942translocation}, OT has undergone substantial theoretical and computational developments.
A key milestone was the dynamical formulation introduced by Benamou and Brenier \cite{benamou2000computational}, which reformulates the static transport problem as a fluid dynamics problem by minimizing the kinetic energy of a time-dependent density subject to a continuity equation. 
In many modern applications, the objects being compared are inherently discrete, such as empirical measures supported on finitely many samples, pixels, particles, or states. In these settings, node-wise distributions of mass or activity evolve over networks and are shaped by the underlying network connectivity, which restricts how mass may be redistributed across nodes, as commonly encountered in social and brain network analysis \cite{Fornito2016BrainNetworks,Newman2013NetworkScience}. Such structural constraints are naturally and efficiently encoded by weighted graphs, where edges specify admissible interactions and weights quantify their strength.
This graph-based perspective leads to OT on graphs \cite{maas2011gradient}, which compares and transports probability measures supported on nodes while respecting the underlying connectivity and edge weights.

In this work, we focus on the computation of the dynamical optimal transport (OT) problem \eqref{eq:graph-bb}
on a finite, connected, simple graph $G=(V,E,\Omega)$ to capture the interpolation trajectories (geodesics) between graph signals.
Here, $V=\{1 \le i\le N\}$ denotes the vertex (or node) set, 
$E$ denotes the edge set, and
$\Omega \in \mathbb{R}^{N \times N}$ is a symmetric weight matrix:
$\Omega=\{\omega_{jl}\}_{j,l=1}^N$ such that $\omega_{jl}=\omega_{lj}>0$  if $(j,l)\in E$, $\omega_{jj}=0$, and $\omega_{jl}=\omega_{lj}=0$ otherwise. 
The dynamical OT problem consists in minimizing the following 
functional
\begin{equation}\label{eq:graph-bb}
\inf_{\rho,\,v}\left\{\int_{0}^{1} \mathcal{K}(\rho(t),v(t))\,dt \;:\;
\partial_t \rho + \operatorname{div}^{\theta}_{G}\!\big(\rho\, v\big)=0,\;
\rho(0)=\mu,\;
\rho(1)=\nu \right\},
\end{equation}
where the discrete divergence operator  $\operatorname{div}^{\theta}_{G}$ is defined in \eqref{the def of divergence}.
The Lagrangian cost functional $\mathcal{K}$ is typically given by the weighted inner product $\langle\cdot,\cdot\rangle_{\theta(\rho)}$,  i.e., $\mathcal{K}(\rho,v)=\langle v, v\rangle_{\theta(\rho)}$, which  relies on a density-dependent weight function $\theta(\rho)$; see Section~\ref{operators on graph}.  The OT problem on graph \eqref{eq:graph-bb} aims to seek a Wasserstein geodesic on discrete graph \cite{CHLZ12,MR4405488, doi:10.1137/21M142160X}, by finding the optimal 
time-dependent density $\rho(t)$ and optimal flux $v(t)$ that connect two probability measures $\mu$ and $\nu$ by minimizing the averaged Lagrangian cost functional.

The dynamical formulation on graphs \eqref{eq:graph-bb} is closely related to the continuous Benamou--Brenier formulation of optimal transport. This relation is particularly transparent when the graph represents a spatial discretization of an underlying domain: nodal masses approximate continuous densities, while edge fluxes act as discrete analogues of momentum variables. Several discrete OT constructions make this connection explicit and analyze their relation with the continuous Wasserstein geometry through discrete Benamou--Brenier formulations and scaling-limit results \cite{connectionwithcontinuous}.
From a computational viewpoint, this parallel also explains why numerical methods for discrete and graph-based OT often build on techniques first developed for the continuous problem, including proximal splitting schemes \cite{Papadakis2014OptimalTransport} and entropic regularization \cite{cuturi2013sinkhorn}.

Within the discrete setting, an extensive body of work has investigated transport problems on networks. In particular, in the absence of edge-capacity constraints, graph optimal transport is closely related to the classical minimum-cost flow problem and admits a linear programming formulation. Classical algorithms for this setting are surveyed in \cite{AhujaMagnantiOrlin1993NetworkFlows}, including cycle-canceling methods, network simplex methods, and Ford--Fulkerson-type approaches. More recently, discrete geometric frameworks for transport-type metrics on graphs have been developed, most notably by Maas \cite{maas2011gradient} and Mielke \cite{mielke2011formulation}. 
To characterize the dynamical evolution of the density trajectory and velocity field for OT problems on graphs, Chow et al. \cite{CHLZ12, li2018computational} introduced a discrete analogue of the Benamou--Brenier dynamical formulation on graphs, representing transport paths via vertex-based densities coupled with edge-based fluxes through a discrete continuity equation.

Yet, numerical methods for dynamical OT on graphs are still relatively underdeveloped. 
A recent work
\cite{doi:10.1137/21M142160X} reformulates the problem as a two-point boundary value problem and applies multiple shooting with Newton's method, but this strategy entails inherent trade-offs. Multiple shooting is attractive because it stabilizes challenging boundary-value problems by decomposing them into a collection of short-time initial-value problems. However, this comes at the cost of enlarging the overall system and limiting the ability to exploit sparsity in the underlying structure. In addition, current methods do not systematically address the non-uniqueness arising from graph cycles and typically assume a priori that the density is positive, rather than enforcing positivity through a robust numerical mechanism.

On the other hand, leveraging a fully coupled Newton solver in the graph setting is nontrivial and introduces several algorithmic difficulties. First, the graph-based dynamical formulation yields a large-scale nonlinear system whose structure is tightly constrained by the topology of the underlying network. These constraints become especially pronounced on graphs containing cycles, where the edge-based velocity variables are no longer independent, potentially leading to a rank-deficient or ill-conditioned Jacobian. Second, while positivity of the density in the continuous setting can be enforced through characteristic-based arguments, such geometric mechanisms do not directly carry over to finite graphs. A finite graph lacks the length-space structure that underpins characteristic flows, making the preservation of density positivity a central requirement for both physical consistency and numerical stability.

To overcome these difficulties, we develop a structure-preserving Newton method tailored to graph topology.
By exploiting the graph topology through a spanning-tree parameterization, the method avoids the redundancy caused by cycles and produces a sparse linear system that can be solved efficiently at each Newton step. An upwind discretization is further used to preserve positivity of the density, which is crucial for obtaining physically meaningful numerical solutions. Furthermore, we demonstrate that solutions computed under different gauge fixings agree up to a cycle-space gauge transformation within the target numerical accuracy, and their discrepancy can be controlled by the time steps and tolerance of the Newton iteration in Remark~\ref{the discrepancy of the spanning tree}.

Beyond validating the proposed Newton solver, the numerical experiments in this work also illustrate several practical application scenarios of dynamical optimal transport on graphs. First, the computed geodesic trajectories provide a natural tool for interpolating and comparing graph signals, which is useful in data analysis tasks where observations are distributed over networked structures, such as sensor networks, traffic systems, and biological interaction networks. Second, the associated edge fluxes reveal dominant transfer pathways and transient concentration patterns, making the method relevant for identifying information propagation routes in social networks, signal transmission in brain connectivity networks, and redistribution mechanisms in infrastructure or communication networks. Third, because the formulation respects graph topology and edge weights, it can serve as a physically meaningful model for constrained transport processes on discrete structures, including load balancing, resource allocation, and diffusion-driven reconfiguration on networks. Finally, the ability to compute stable Wasserstein geodesics on general weighted graphs may also support downstream machine learning tasks on graph-structured data, such as graph signal interpolation, clustering, generative modeling, and shape or distribution comparison when the underlying geometry is inherently discrete.

The remainder of this paper is organized as follows.
In Section~\ref{Newton's Method for OT Problem on Graphs}, we review the basic operators on graphs and derive the formulation of the Wasserstein geodesic equations. 
In subsection~\ref{Newton's iteration for OT Problem on Graphs}, a spanning-tree-based gauge fixing approach is introduced to eliminate the redundancy in the velocity variables, leading to a reduced nonlinear system.
The convergence of the resulting Newton iteration is analyzed in Section~\ref{the convergence of Newton iteration}.
Several numerical experiments are presented in Section~\ref{Numerical Examples} to validate the proposed method and several applications of the proposed algorithm are also discussed.

\section{Finite-difference type Newton's iteration on graphs} 
\label{Newton's Method for OT Problem on Graphs}

In this section, we present the overall workflow for computing dynamical optimal transport on a graph.
Starting from the graph-based Benamou--Brenier formulation, we first introduce the discrete differential operators and the mobility-weighted kinetic energy that define the action functional on $G$.
We then apply a finite-difference discretization in time, which converts the continuous-in-time optimality conditions into a coupled nonlinear algebraic system for the discrete densities and edge fluxes at all time levels. To solve this system, we employ Newton’s method, whose Jacobian exhibits a sparse block structure induced by the discrete continuity constraints on the graph.

\subsection{Discrete OT problem on graphs}
\label{operators on graph}
Throughout this paper, we  consider an undirected, connected, simple graph $G=(V, E,\Omega)$; recall that a graph is simple if there are no self-loops or multiple edges (see e.g. \cite{MR2035186}). Below,  we denote
$\mathcal N(i)=\{j\in V: (i,j)\in E\}$ as the  set of neighbors of the node $i$.

Denote the set of discrete probabilities on the graph by ${\mathcal{P}}(G)$:
$$
\mathcal P(G)=\Big\{(\rho_i)_{i=1}^N\ |\, \sum_{i=1}^N \rho_i =1, \; \rho_i\ge 0 \Big\},
$$
and $\mathcal P_o(G)$ is its interior (i.e., all  $\rho_i> 0$ for $i\in V$).
A vector field $v$ on $E$ is a $N\times N$ skew-symmetric matrix, and the weighted inner product of two vector fields $u,v$ is defined by
\[
\langle u,v\rangle_{\theta(\rho)}
=
\frac{1}{2}\sum_{(j,l)\in E} u_{jl}v_{jl}\theta_{jl}(\rho).
\]
where the coefficient 1/2 is taken since every edge in $G$ is counted twice.  
Here, the weight function
\(\theta:\mathbb{R}_{\ge 0}^2 \to \mathbb{R}_{\ge 0}\) 
is defined by
\[
\theta_{ij}(\rho) := \theta(\rho_i,\rho_j), \qquad (i,j)\in E,
\]
where \(\theta\) satisfies the following properties:
\begin{enumerate}
\renewcommand{\labelenumi}{(\roman{enumi})}
\item \textbf{Maximum principle:}
      \(\min\{a,b\}\le \theta(a,b)\le \max\{a,b\}\).
\item \textbf{Positive $1$-homogeneity:}
      \(\theta(\lambda a,\lambda b)=\lambda\,\theta(a,b)\) for all \(\lambda>0\).
\end{enumerate}

The density dependent weight $\theta_{ij}(\rho)$ on the edge $(i,j)\in E$ was introduced in  \cite{maas2011gradient,CHLZ12,CLZ19b} to capture the 
long-time dynamics for Fokker-Planck equations on graphs and the dispersion relationship for the Schr\"odinger equation 
on graphs.  
The standard example satisfying the above properties is the arithmetic mean
\begin{equation*}\label{eq:theta_def}
    \theta_{jl}({\rho}) := \frac{{\rho}_j + {\rho}_l}{2},
    \qquad l \in \mathcal{N}(j) .
\end{equation*} 
For other possible choices of $\theta$, we refer to \cite{MR4405488}.

We introduce the discrete divergence of the flux function $\rho v$ as
\begin{align}
\operatorname{div}_{G}^{\theta}(\rho v)=(\sum_{j\in \mathcal N(i)}\sqrt{\omega_{ij}}v_{i,j}\theta_{ij})_{i=1}^N.
\label{the def of divergence}
\end{align}
Accordingly, the gradient operator $\nabla_G$ on the graph is given through the following dual formulation  
$$-\langle \operatorname{div}_G^{\theta}(\rho v), \widetilde{\Phi} \rangle
=
\langle v, \nabla_G \widetilde{\Phi} \rangle_{\theta(\rho)},$$
where $\widetilde \Phi$ is a function defined on every node of the graph. It can be seen that $\nabla_G$ maps a node-defined function to a vector field on the edges. 

In this paper, we consider the following optimal transport problem 
on the graph: 
\begin{align}\label{dis-ot}
W(\rho(0),\rho(1))^2
&= \inf_{v(t),\,\rho(t)}
\Biggl\{
\int_{0}^1 \langle v(t),v(t)\rangle_{\theta(\rho(t))}\,dt
:\notag\\
&\qquad
\frac{d\rho(t)}{dt}
+ \operatorname{div}_G^{\theta}\bigl(\rho(t) v(t)\bigr)=0,
\quad \rho(0)=\mu,\quad \rho(1)=\nu
\Biggr\}.
\end{align}
For simplicity, we denote 
$$\rho(t)=(\rho_i(t))_{i=1}^N,\;  v(t)=(v_{i,j}(t))_{(i,j)\in E},\; \text{for}\; t\in [0,1].$$ 
By letting $S$ such that $v_{i,j}=(S_j-S_i)\sqrt{\omega_{ij}}$ (cf. $\nabla S=v$ in \cite{MR4405488}), the critical point of \eqref{dis-ot} satisfies the discrete Wasserstein geodesic equation (see \cite{doi:10.1137/21M142160X,MR4405488})
\begin{equation}\label{dhs}\begin{split}
&\frac {d\rho_i}{d t}=\sum_{j\in \mathcal N(i)}\omega_{ij}(S_i-S_j)\theta_{ij}(\rho)=\frac {\partial H}{\partial S_i},\\
&\frac {d S_i}{dt}=-\frac 12\sum_{j\in \mathcal N(i)}\omega_{ij}  (S_i-S_j)^2 \frac {\partial \theta_{ij}(\rho)}{\partial \rho_i}=-\frac {\partial  H}{\partial \rho_i}+C(t)
\end{split}\end{equation}
with boundary conditions $\rho(0)=
\mu$ and $\rho(1)=\nu.$ 
We observe that \eqref{dhs} is a Wasserstein Hamiltonian system on the graph \cite{MR4405488}, whose phase flow preserves the 
discrete Hamiltonian 
$${H}(\rho,S)=\frac 14\sum_{i=1}^N\sum_{j\in \mathcal N(i)}|S_i-S_j|^2\omega_{ij} \theta_{ij}(\rho).$$ 
Here the potential $S_i$ induces the optimal velocity $v_{i,j} =({S_j-S_i})\sqrt{\omega_{ij}}$ for $(i,j) \in E$, and $C(t)$ is a function depending only on $t$.

\subsection{A finite-difference Newton scheme with spanning-tree reduction}
\label{Newton's iteration for OT Problem on Graphs}
In this section, we present a Newton framework for solving \eqref{dhs}.
As shown in  \cite{CLZ21u}, we can rewrite \eqref{dhs} in the equivalent  fluid dynamics formulation
\begin{equation}
\begin{aligned}
& \displaystyle\frac{d \rho_i}{dt}= -\sum_{j \in \mathcal N(i)} \displaystyle \sqrt{\omega_{ij}} v_{i,j} \theta_{ij} (\rho),       \\
& \displaystyle\frac{d v_{i,j}}{dt}= -\displaystyle\frac{1}{2} \sum_{k \in \mathcal N(j) } \sqrt{\omega_{ij}} (v_{k,j})^2 \displaystyle\frac{ \partial \theta_{jk}(\rho)}{\partial \rho_j}+\displaystyle\frac{1}{2} \sum_{k \in \mathcal N(i) } \sqrt{\omega_{ij}}  (v_{k,i})^2 \displaystyle\frac{ \partial \theta_{ik}(\rho)}{\partial \rho_i}. 
\end{aligned}
\label{equation expression}
\end{equation}  However, there is redundant information of the velocities in \eqref{equation expression} since the optimal velocity is induced by the potential, $v_{i,j} =({S_j-S_i})\sqrt{\omega_{ij}}$. Here, the number of unknown $S_i$ is at most $N-1$ instead of $(N-1)\times (N-1)$. 
To eliminate this redundancy, we make systematic use of a spanning-tree structure of the underlying undirected graph. We first recall the notion of a spanning tree.

\begin{definition}[Spanning tree]
 Let $G=(V,E,\Omega)$ be a connected undirected weighted simple graph.
A \emph{spanning tree} of $G$ is a subgraph $T=(V,E_T)$ with $E_T\subseteq E$
that is connected and acyclic and $|E_T|=|V|-1$.
\end{definition}

We note that, see \cite[Proposition~1.5.6] {Diestel}, any connected undirected graph  $G$ admits a spanning tree $T=(V,E_T)$ by fixing a edge orientation.  So, we may assume that we have a spanning tree $T$ and an orientation of its edges.
The next lemma shows that, for potential-induced velocities, every edge velocity on $G$ can be expressed as a signed combination of the tree-edge velocities along the unique tree path.

\begin{lemma}\label{lem:tree-representation}
Let $G=(V,E,\Omega)$ be a connected undirected graph and let $T=(V,E_T)$ be a spanning tree of $G$ with a fixed orientation.  
Consider the optimal  velocity variables $\{v_f\}_{f\in E_T}$ of \eqref{equation expression} on the tree.  
Then, for any edge $e=(i,j)\in E$, there exist coefficients $a_f\in\{-1,0,1\}$, $f\in E_T$, such that
\begin{equation}
    \frac 1{\sqrt{\omega_e}}\, v_e
    = \sum_{f\in E_T} a_f\, \frac 1{\sqrt{\omega_f}}\, v_f .
\end{equation}
\end{lemma}

\begin{proof} 
 For any edge $e=(i,j)\in E$,
  there is a unique simple path in $T$ from $i$ to $j$. Using   the definition $v_{i,j}=(S_j-S_i)\sqrt{\omega_{ij}},$ we have that
  there exists coefficient $a_f\in \{-1,1\}$, such that 
\begin{align}\label{useful-express}
\frac 1{\sqrt{\omega_{e}}}v_{e}=\sum_{f\in E_T}a_f \frac 1{\sqrt{\omega_{f}}} v_f.  
\end{align}
\end{proof}

Because of the above, we will focus on solving the following reduced system
\begin{equation}
\label{equation expression1}
\begin{aligned}
\frac{d \rho_i}{dt}
&= - \sum_{j \in \mathcal N(i)} \sqrt{\omega_{ij}}\, v_{i,j}\, \theta_{ij}(\rho),
\qquad i \le N-1, \\
\frac{d v_{i,j}}{dt}
&= -\frac{1}{2} \sum_{k \in \mathcal N(j)} \sqrt{\omega_{ij}}\, (v_{k,j})^2
   \frac{\partial \theta_{jk}(\rho)}{\partial \rho_j} \\
&\quad
   + \frac{1}{2} \sum_{k \in \mathcal N(i)} \sqrt{\omega_{ij}}\, (v_{k,i})^2
   \frac{\partial \theta_{ik}(\rho)}{\partial \rho_i},
\qquad (i,j) \in E_T, \\
\frac{1}{\sqrt{\omega_e}} v_e
&= \sum_{f \in E_T} a_f \frac{1}{\sqrt{\omega_f}} v_f,
\qquad e \in E.
\end{aligned}
\end{equation}
satisfying the mass conservation $\sum\limits_{i\in V}\rho_{i}=1.$ Here, $\mathcal N(i)$ denotes the set of neighbors of node $i$ in the original graph $G=(V,E,\Omega)$, rather than the corresponding neighbor set in the spanning tree alone. Accordingly, although the evolution equations for the velocity variables are imposed only on the tree edges $(i,j)\in E_T$, all summations appearing above are taken over the full graph. In particular, for any non-tree edge $e\in E\setminus E_T$, the associated velocity $v_e$ is determined from the tree-edge velocities via the last relation.

Next, we introduce a Newton method for solving the discrete OT problem \eqref{dis-ot} through the nonlinear system associated with \eqref{equation expression1}. We first consider the exact time-integrated nonlinear system, denoted by $\tilde F$, which is obtained by integrating the continuous dynamical system \eqref{equation expression1} exactly over each time interval $[t_m,t_{m+1}]$. 
Its unknowns are the nodal samples of the continuous solution,
\[
\bar\rho_i^m := \rho_i(t_m), \qquad \bar v_f^m := v_f(t_m),
\]
for $1\le i\le N-1$, $f\in E_T$, and $1\le m\le M+1$. 
Although these unknowns are discrete in number, they represent values of the continuous solution at the time nodes. 
The evolution over each subinterval is still governed by the continuous dynamics, so that the values inside $[t_m,t_{m+1}]$ are determined by exact integration starting from the nodal data at $t_m$.

Accordingly, the exact residual $\tilde F$ is defined by
\begin{align*}
\tilde F_i^m
&:=
\bar\rho_i^{m+1}-\bar\rho_i^m
+\tilde\Phi_i\Bigl(t_m,t_{m+1},(\bar\rho_i^m)_{i\le N-1},(\bar v_f^m)_{f\in E_T}\Bigr), \qquad m\le M,\\
\tilde F_f^m
&:=
\bar v_f^{m+1}-\bar v_f^m
+\tilde\Phi_f\Bigl(t_m,t_{m+1},(\bar\rho_i^m)_{i\le N-1},(\bar v_f^m)_{f\in E_T}\Bigr), \qquad m\le M,
\end{align*}
subject to the constraint \eqref{useful-express}, the mass conservation condition $\bar\rho^m=(\bar\rho_i^m)_{i\in V}\in\mathcal P_o(G)$, and the boundary conditions $\bar\rho^1=\mu$ and $\bar\rho^{M+1}=\nu$. 
Here, with $\tau=t_{m+1}-t_m$, the exact transition terms are
\[
\tilde \Phi_i \Bigl(t_m,t_{m+1},(\bar{\rho}_i^m)_{i\le N-1},(\bar{v}_f^m)_{f\in E_T}\Bigr)
:=
\int_{t_m}^{t_{m+1}}
\left(
\sum_{j \in \mathcal N(i)} \sqrt{\omega_{ij}}\, v_{i,j}(t)\,\theta_{ij}(\bar\rho(t))
\right)\,dt,
\]
and
\begin{align*}
\tilde \Phi_f\Bigl(t_m,t_{m+1},(\bar{\rho}_i^m)_{i\le N-1},(\bar{v}_f^m)_{f\in E_T}\Bigr)
:=
&\frac12\int_{t_m}^{t_{m+1}}
\Biggl(
\sum_{k \in \mathcal N(j)} \sqrt{\omega_{ij}}\, (v_{k,j}(t))^2
\frac{\partial \theta_{jk}(\bar{\rho}(t))}{\partial \rho_j}\\
&-
\sum_{k \in \mathcal N(i)} \sqrt{\omega_{ij}}\, (v_{k,i}(t))^2
\frac{\partial \theta_{ik}(\bar{\rho}(t))}{\partial \rho_i}
\Biggr)\,dt.
\end{align*}

Since the exact transition operators $\tilde\Phi_i$ and $\tilde\Phi_f$ are generally unavailable in closed form, in practice we replace them by numerical quadrature approximations. 
This leads to the fully discrete nonlinear system $F$, whose unknowns are denoted by $\rho_i^m$ and $v_f^m$ and represent the numerical approximations to $\bar\rho_i^m$ and $\bar v_f^m$, respectively. 
For example, using the left-endpoint rule, we define
\begin{align}\label{eq:Ftilde-rho}
F_i^m&:=
\rho_i^{m+1}-\rho_i^m
+\Phi_i\Bigl(t_m,t_{m+1},(\rho_i^m)_{i\le N-1},(v_f^m)_{f\in E_T}\Bigr), \qquad m\le M,\\
\label{eq:Ftilde-v}
F_f^m&:=
v_f^{m+1}-v_f^m
+\Phi_f\Bigl(t_m,t_{m+1},(\rho_i^m)_{i\le N-1},(v_f^m)_{f\in E_T}\Bigr), \qquad m\le M,
\end{align}
again together with \eqref{useful-express}, the constraint $\rho^m=(\rho_i^m)_{i\in V}\in\mathcal P_o(G)$, and the boundary conditions $\rho^1=\mu$ and $\rho^{M+1}=\nu$. 
Newton's method is then applied to solve this computable residual system $F=0$.

In this work, for $\Phi_i$ and $\Phi_f$ in \eqref{eq:Ftilde-rho}-\eqref{eq:Ftilde-v},
we use the left-rectangular rule, i.e.
\begin{align}\label{LRR}
   \Phi_i &= 
        \sum_{j \in \mathcal N(i)} \displaystyle \sqrt{\omega_{ij}} v_{i,j}^m \theta_{ij} (\rho^m)\tau,
 \\
    \Phi_f &=\displaystyle\frac{1}{2} \sqrt{\omega_{ij}} \tau \Big[ \sum_{k \in \mathcal N(j) }  (v_{k,j}^m)^2 \displaystyle\frac{ \partial \theta_{jk}(\rho^m)}{\partial \rho_j}-\sum_{k \in \mathcal N(i) }   (v_{k,i}^m)^2 \displaystyle\frac{ \partial \theta_{ik}(\rho^m)}{\partial \rho_i}\Big], (i,j)\in E_T. \label{phi_f}
\end{align}

For convenience,  we rewrite the 
function \eqref{eq:Ftilde-rho}-\eqref{eq:Ftilde-v} in a compact form 
{\small
\begin{equation} 
 \mathbf F_{x} =
\begin{pmatrix}
\mathbf {F}_{\hat \rho} \\[4pt]
\mathbf {F}_{{\hat v}_T}
\end{pmatrix}
\approx 
\begin{pmatrix}
0 \\[4pt]
0    
\end{pmatrix}, \;x={(\hat \rho,{\hat v_T})}:=((\rho^{m}_i)_{2\le m \le M,i\le N-1}, (v^{m}_f)_{m\le M+1,f\in E_T}), 
\label{the def of full system tilde F}
\end{equation}}
where
\begin{equation}\label{eq:Ftilde-rho1}
\mathbf {F}_{\hat \rho}
:=
\big[
    (\mathbf {F}_{i}^1)_{i\le N-1},
    \ldots,
    ( \mathbf {F}_{i}^M)_{i\le N-1}
\big]^{\top},
\end{equation}
and
\begin{equation}\label{eq:Ftilde-v1}\mathbf {F}_{{\hat v}_T}
:=
\big[(\mathbf {F}_f^1)_{f\in E_T}, \ldots, ({\mathbf {F}}_f^{M})_{f \in E_T} \big]^\top.
 \end{equation}
We also denote the value function of the exact solution by $\widetilde {\bf F}_x$ when $\Phi_i$ and $\Phi_f$ are the exact integrals. Note that for the continuous OT problem, the exact solution of $\widetilde {\bf F}$ is   $\bar{x}=((\bar {\rho}_i^m)_{2\le m\le M,i\le N-1} $, $(\bar {v}_f^m)_{m\le M+1,f\in E_T}),$ at the temporal grids. Here {$((\bar {\rho}_i(t))_{i\le N-1},(\bar {v}_f(t))_{f\in E_T})$}, $t\in [0,1],$ is the exact solution of \eqref{equation expression1}. 

By formulating the Newton iteration method for solving the nonlinear system, we have 
{\begin{equation}
    x^{(k+1)} = x^{(k)} - \left( J_{\bf{F}}(x^{(k)}) \right)^{-1} {\bf{F}}(x^{(k)}),
    \label{Newton iteration}
\end{equation}where \( x^{(k)}=\big((\rho^m_i)^{(k)}_{2\le m\le M,i\le N-1}, (v_f^m)_{m\le M+1,f\in E_T}^{(k)}\big)\) is the current approximation at the \(k\)-th iteration, \( J_{\bf F}(x^{(k)}) \) (also denoted by $J^{(k)}$ for simplicity) is the Jacobian matrix 
evaluated at the \(k\)-th iterate. }
In our experiments, we enforce the following stopping criterion for Newton's iteration: 
\begin{equation}\label{NewtStops}
\|{\bf F}(x^{(k)})\|< {\rm tolerance} : = \varepsilon.
\end{equation}

\begin{algorithm}
\caption{Optimal Transport Solver on Graph}\label{alg:OT-solver}
\begin{algorithmic}[1]
    \State \textbf{Input:} an undirected, connected graph $G=(V,E,\Omega)$, $|V|=N$;
    time grid points $\{t_m\}_{m=1}^{M+1}$ with $t_1=0$, $t_{M+1}=1$;
    initial distribution $\mu$, target distribution $\nu$;
    max-number of Newton iterations $\mathrm{Maxits}$;
    numerical integration $(\Phi_i,\Phi_f)_{i\le N-1,\;f\in E_T}$.

    \State Find a spanning tree $(V,E_T)$ of the graph $G$.
    
    \State Give an initial guess
    $x^{(0)}=((\hat{\rho})^{(0)},(\hat{v}_f)^{(0)})$.

    \For{$k=1,\ldots,\mathrm{Maxits}$}
        \If{\eqref{NewtStops} is satisfied}
            \State stop
        \EndIf
        \State Solve $J^{(k)}d^{(k)}=-\mathbf{F}(x^{(k)})$, and set 
        $x^{(k+1)}=x^{(k)}+d^{(k)}$.
    \EndFor
    
    \State Output of Newton's iteration
    \[
    \bar{x}
    =
    ((\bar {\rho}_i^m)_{2\le m\le M,\;i\le N-1},
    (\bar {v}_f^m)_{m\le M+1,\;f\in E_T}).
    \]
    
    \State Recover the velocities on all edges:
    for each $e \in E$, compute
    \[
    \bar{v}^m_e
    =
    \sqrt{\omega_e}
    \sum_{f\in E_T} a_{f}
    \frac{1}{\sqrt{\omega_f}}\,\bar{v}^m_f, 
    \qquad m=1,\dots,M+1.
    \]

    \State \textbf{Output:} the discrete optimal transport solution, consisting of
    the nodal densities
    $\{(\bar{\rho}^m_i)\}_{i\in V,\;m=1,\dots,M+1}$
    and the recovered edge velocities
    $\{(\bar{v}^m_e)\}_{e\in E,\;m=1,\dots,M+1}$.
\end{algorithmic}
\end{algorithm}

\begin{remark}
{To construct a spanning tree, several  
algorithms may be employed, such as Kruskal's or Prim's algorithm for minimum spanning trees, or traversal-based methods like Breadth-First Search (BFS) and Depth-First Search (DFS) in \cite{cormen2022introduction}}. In our work, spanning trees are computed using Kruskal’s algorithm.
We also note that our algorithm is robust to the choice of spanning tree. As we will see
in Proposition \ref{prop:numerical error}, the discrepancy between the numerical solution associated with any spanning tree and the exact solution is bounded by the global error of the numerical method. 
Meanwhile,  using \eqref{useful-express} and the mass conservation law, one can recover the solution  at temporal grids on all nodes and edges of $G$.

When $G$ is a lattice graph arising from a spatial semi-discretization of a continuous domain, the edge weights and neighborhood structure take a particularly simple form. In this case, the standard finite-difference scaling gives $\omega_{ij}=\delta_x^{-2}$,
where $\delta_x$ denotes the spatial step size, and the adjacency sets $\mathcal N(i)$ are determined by the local stencil of the grid (e.g., nearest neighbors). For a two-dimensional lattice graph, a convenient spanning tree can be constructed by taking (i) all edges along one fixed coordinate direction (say, the $x$-direction) and (ii) for each fixed $x$-index, all edges along the $y$-direction. This choice yields a simple, structured spanning tree that is well aligned with the grid geometry; see, e.g., \cite{CLZ21u}.
\end{remark}

\begin{remark}
To alleviate the substantial cost of assembling the Jacobian at every Newton step, we employ a chord-type quasi-Newton approach in selected experiments. The Jacobian is evaluated once at the initial iterate and kept fixed thereafter. Although this modification degrades the local convergence from quadratic to linear, it avoids repeated Jacobian assembly and thus substantially reduces the per-iteration cost, improving overall runtime efficiency for the problems considered.

Unlike the multiple shooting method of \cite{doi:10.1137/21M142160X},
we first discretize \eqref{equation expression1} in time and then apply Newton’s method to solve the resulting nonlinear algebraic system \eqref{eq:Ftilde-rho}–\eqref{eq:Ftilde-v}. This formulation avoids time-interval partitioning and the associated matching conditions required by multiple shooting. As demonstrated by comparisons in Tables~\ref{tab:cpu_1d2d} and \ref{tab:comparison}, for the prescribed initial and terminal densities, the proposed Newton-based finite-difference solver consistently achieves improved computational efficiency. Moreover, when an upwind weight $\theta_{ij}$ is adopted, the resulting scheme is positivity-preserving for the density at the discrete level (see subsection \ref{sec:positivity}).
\end{remark}

\section{Properties of Newton's method for OT problem  on  graph}
\label{the convergence of Newton iteration}

In this section, we analyze the convergence properties and numerical integration errors of the proposed Newton method for the optimal transport problem on graphs.

\subsection{Convergence of Newton's method}
We next establish the local quadratic convergence of Newton’s method for the nonlinear algebraic system induced by the time discretization.  As in standard convergence results for Newton's method, under closeness assumptions of the initial guess to the exact solution, the task is to show that the Jacobian matrix at the exact solution is invertible.

\begin{theorem}\label{main-thm}   
   Let $\bigl((\bar{\rho}_i(t))_{i\le N-1},(\bar v_f(t))_{f\in E_T}\bigr), 
\quad t\in[0,1]$, be the unique solution of~\eqref{equation expression1} which belongs to
\( C^{1}\bigl([0,1];  \mathbb{R}^{2N-2}\bigr),\)
and satisfies that 
\(
\min\limits_{t\in[0,T]} \min\limits_{i=1}^N \bar{\rho}_i(t) \ge c > 0,\)
and suppose that 
\(\frac{\partial (\bar{\rho}_i(t_{M+1}))_{i\le N-1}}{\partial (\bar{v}_f(t_1))_{f\in E_T}}\)
is invertible.  Here $t_1=0, t_{m}=(m-1)\tau, m\le M+1$, such that $t_{M+1}=1.$

Assume that we form $\mathbf F$ in \eqref{eq:Ftilde-rho}-\eqref{eq:Ftilde-v} using \eqref{LRR} and \eqref{phi_f}. Also assume that the initial vector $x^{(0)}$ for Newton's method applied to the exactly computed $\tilde{\mathbf F}$ is sufficiently close to $\bar{x} :=((\bar{\rho}_i(t_m))_{2\le m\le M,i\le N-1},(\bar{v}(t_m))_{m\le M+1,f\in E_T})^{\top},$
i.e., $\lvert x^{(0)} - \bar{x} \rvert = \mathcal{O}(\eta)$ for $\eta > 0$ sufficiently small.
Then Newton's method with $\mathbf F$ being replaced by $\tilde{\mathbf F}$  is quadratically convergent to $\bar{x}.$ 
\end{theorem}

\begin{proof}
Since we are using the left-rectangle rule \eqref{LRR} for numerical integration,
it holds that \begin{align*}
 0=\widetilde {\bf F}(\bar{x})= {\bf F}(\bar{x})+\mathcal O((\tau\sqrt{\omega_{\max}})^{2}),
\end{align*}
where $\omega_{\max}=\max_{ij\in E}|\omega_{ij}|.$

To prove the quadratic convergence of Newton's method, it suffices to show that $J_{\widetilde{\mathbf F}}(x)$ is invertible and Lipschitz in a small neighborhood of $\bar{x}$. 
According to our assumption, $\widetilde {\bf F}, \bf F$ are   continuously  differentiable w.r.t. $x$, and thus are both Lipschitz in a small ball centered at $\bar{x}$ of radius $\eta$, $\mathcal B(\bar{x},\eta)$, with $\eta$ small enough so that 
$x=\big((\rho^m_i)_{2\le m\le M,i\le N-1}, (v_f^m)_{m\le M+1,f\in E_T}^{}\big)^{\top}$ satisfies $\min_{i\le N-1,m\le M}\rho_{i}^m>0$ for any $x\in \mathcal B(\bar{x},\eta)$.   So, we have to prove invertibility of $J_{\widetilde {\bf F}}$ near $\bar{x}$ so that the iterates $x^{(k+1)}$ in  \eqref{Newton iteration} remain in $\mathcal B(\bar{x},\eta)$ if $x^{(k)}$ is in $\mathcal B(\bar{x},\eta)$, for sufficiently small $\eta>0$.

Denote $\mathbf {F}_{\hat \rho}^{\rm fin}=( F_{i}^M)_{i\le N-1},$  and $
\mathbf F_{\hat \rho}=((\mathbf F_{\hat \rho}^{\rm int})^{\top}, (\mathbf F_{\hat \rho}^{\rm fin})^{\top})^{\top}.$ 
We represent  
the Jacobian matrix \( J_{\tilde{\bf F}}(x) \) as the following block matrix:
{\small
\begin{equation}
    J_{\tilde{
    \bf F}}(x)
    = 
    \frac{\partial (({\mathbf F}_{\hat\rho}^{{\rm int}})^{\top},\,  (\mathbf F_{\hat \rho}^{{\rm fin}})^{\top},\, ({\textbf{F}}_{{\hat v}_T})^{\top})^{\top}}{\partial (\hat \rho^\top,{{\hat v}_T}^\top)^\top } \Big|_{x} + \mathcal{O}(\tau^2 \omega_{\max})
    \equiv
    \begin{bmatrix}
        A & B \\[4pt]
        C & D
    \end{bmatrix}+ \mathcal{O}(\tau^2 \omega_{\max}),
    \label{the def of Newton JAcobi}
\end{equation}}
where $\hat{\rho}$ and ${\hat v}_T$ are defined in \eqref{the def of full system tilde F} and the block matrices are defined by
{
\begin{align}
\label{the definition of matrix A}
    A &= \frac{\partial {\bf F}_{\hat \rho}^{{\rm int}}}{\partial \hat \rho} \in \mathbb{R}^{m_1 n_1 \times m_1 n_1}, & 
    B &= \frac{\partial {\bf F}_{\hat \rho}^{\mathrm{int}}}{{{\partial\hat v}_T}} \in \mathbb{R}^{m_1 n_1 \times m_2 n_1}, \\\nonumber
    C &= \frac{\partial (({\bf F}_{\hat\rho}^{\mathrm{fin}})^{\top},\, ({\mathbf F}_{{\hat v}_T})^{\top})^{\top}}{\partial \hat \rho} \in \mathbb{R}^{m_2 n_1 \times m_1 n_1}, & 
    D &= \frac{\partial (({\bf F}_{\hat \rho}^{\mathrm{fin}})^{\top},\, ({\bf F}_{{\hat v}_T})^{\top})^{\top}}{\partial {{\hat v}_T}} \in \mathbb{R}^{m_2 n_1 \times m_2 n_1},
\end{align}}
where $m_1 = M-1$, $m_2 = M+1$, and $n_1 = N-1$.

On the one hand, since we are using \eqref{LRR}
$A=I$ and thus $\frac {\partial \widetilde {\bf F}_{\hat \rho}^{\rm int}}{\partial \hat \rho}$  is invertible for sufficiently small $\tau>0.$ By the implicit function theorem, there exists a function $\mathcal{H}(\cdot)$ such that 
$${\bf\widetilde {F}}_{\widehat \rho}^{\rm int}\Big(\mathcal{H}({{\hat v}_T}),{{\hat v}_T}\Big)=0,$$ which is
defined in 
a small neighborhood of $\mathcal B\big(  \bar{\hat v}_T ,\epsilon'\big)$ with $\epsilon'>0$ sufficiently small. Here ${\bar{\hat v}_T}$ is the velocity component of $\bar x$.
Moreover, 
\begin{align}
\label{implicit-G1}
\frac {d \mathcal{H}}{d  \hat v_T}( \hat v_T)=-A^{-1}B\Big|_{x=(\hat \rho^{\top}, \hat v_T^{\top})^\top}+\mathcal O(\tau^2\omega_{\max}).
\end{align}
On the other hand, in such a small neighborhood, consider the following function
$$ \widetilde {\mathcal{H}}(\hat v_T):=\Big(\big(\mathbf {\tilde F}_{\hat \rho}^{\rm fin}\big({\mathcal{H}}(\hat v_T),{\hat v}_T\big)\big)^{\top}, \big(\mathbf {\tilde F}_{{\hat v}_T} \big({\mathcal{H}}({\hat v}_T),{\hat v}_T\big)\big)^{\top}\Big)^{\top}.$$
For brevity we let 
$\Psi(\hat\rho,\hat v_T) := \bigl((\tilde{\mathbf F}^{\mathrm{fin}}_{\hat\rho}(\hat\rho,\hat v_T))^\top,
(\tilde{\mathbf F}_{\hat v_T}(\hat\rho,\hat v_T))^\top
\bigr)^\top.$ Then, by the chain rule and \eqref{implicit-G1}, it holds that
\begin{align*}
\frac{d\widetilde{\mathcal H}}{d\hat v_T}\Big|_{\hat v_T}
&=
\left.\partial_{\hat v_T}\Psi(\hat\rho,\hat v_T)\right|_{\hat\rho=\mathcal H(\hat v_T)}
+
\left.\partial_{\hat\rho}\Psi(\hat\rho,\hat v_T)\right|_{\hat\rho=\mathcal H(\hat v_T)}
\frac{d\mathcal H}{d\hat v_T}\Big|_{\hat v_T} \\
&=
\left[D-C(A^{-1}B)\right]
\Big|_{x=(\mathcal H(\hat v_T)^\top,\hat v_T^\top)^\top}
+\mathcal O(\tau^2\omega_{\max}).
\end{align*}
Therefore, by the Schur complement, we conclude that 
\begin{align*}
    \det(J_{{\widetilde {\bf F}}})&=\det(A)\,\det(D-CA^{-1}B)+\mathcal O(\tau^2\omega_{\max}N M)\\
    &= \det(D-CB)+\mathcal O(\tau^2\omega_{\max}N M) =\det J_{{\hat v}_T}
\end{align*}
in a small neighborhood of $\bar{v}_T$, where $J_{{\hat v}_T}:=\frac {d \widetilde {\mathcal{H}}}{d {\hat v}_T} \Big|_{{\hat v}_T}$.

It remains to prove the invertibility of $J_{{\hat v}_T}$.
To this end, we decompose the vector ${\hat v}_T=(v_f^m)_{m\le M+1,f\in E_T}$ into the following two parts: 
\({\hat v}_T = 
(
\mathcal V^{\top}, 
\widetilde {\mathcal V}^{\top}
)^{\top},\)
where $\mathcal V$ denotes the velocity vector at the first time level and 
$\widetilde {\mathcal V}$ collects all remaining velocity variables.
In this ordering, the matrix $J_{{\hat v}_T}$ takes the block form
\[
J_{{\hat v}_T}
=
\begin{bmatrix}
A_0 & B_0\\[3pt]
C_0 & D_0
\end{bmatrix}+\mathcal O(\tau^2\omega_{\max}),
\]
where
\begin{equation}
\begin{aligned}
A_0=\frac{\partial  {\bf F}_{\hat \rho}^{\mathrm{fin}}}{\partial \mathcal V}\Big|_{\big({\mathcal{H}}({\hat v}_T)^\top,\, {{\hat v}_T}^\top \big)^\top}\in \mathbb{R}^{n_1 \times n_1},\;
B_0=\frac{\partial  {\bf F}_{\hat \rho}^{\mathrm{fin}}}{\partial \widetilde {\mathcal V}}\Big|_{\big({\mathcal{H}}({\hat v}_T)^\top,\, {{\hat v}_T}^\top\big)^\top}\in \mathbb{R}^{n_1 \times n_1M},\\
C_0=\frac{\partial {\bf F}_{{\hat v}_T}}{\partial \mathcal V}\Big|_{\big({\mathcal{H}}({\hat v}_T)^\top,\, {{\hat v}_T}^\top\big)^\top}\in \mathbb{R}^{ n_1M \times n_1},\;
D_0=\frac{\partial {\bf F}_v}{\partial \widetilde {\mathcal V}}\Big|_{\big({\mathcal{H}}({\hat v}_T)^\top,\, {{\hat v}_T}^\top \big)^\top}\in \mathbb{R}^{n_1 M \times n_1 M}.
\end{aligned}
\label{the def of D_0}
\end{equation}

Since $\Phi_i,\Phi_f$ are left-rectangle numerical integration, $D_0$ is block lower triangular with identity diagonal blocks, and is therefore invertible. By the implicit function theorem again, there exists a differentiable map $\Gamma$
defined on a small neighborhood of $\mathcal V^*=((\bar{v})^{1}_{f})_{f\in E_T}$
such that 
$\widetilde {\mathcal V}=\Gamma(\mathcal V).$

Consider the function defined by 
${U}(\mathcal V):=\tilde{\mathbf F}_{\rho}^{\rm fin}(\mathcal V, \Gamma(\mathcal V))$ in this small neighborhood of $\mathcal V^*$.
The Schur complement  shows that
{\small
\begin{equation}
\frac{\partial U}{\partial \mathcal V} (\mathcal V)
=
[A_0 - B_0D_0^{-1}C_0]\Big|_{({\mathcal{H}}({\hat v}_T)^{\top},{\hat v}_T^{\top})=\big(\mathcal H(\mathcal V, \Gamma( \mathcal V))^{\top},(\mathcal V, \Gamma( \mathcal V))^{\top}\big)}+\mathcal O(\tau^2\omega_{\max}NM)
= \hat J_{\mathcal V}.
\label{the def of hat j}
\end{equation}}
Thus the invertibility of the full Jacobian $J_{\tilde {\bf F}}$ is equivalent to the nonsingularity of the 
matrix $\hat J_{\mathcal V}$.   This implies that  
\[
\det(J_{{\hat v}_T})
=  
\det(D_0)\,\det(A_0 - B_0D_0^{-1}C_0)+\mathcal O(\tau^2\omega_{\max}NM).\]
For sufficiently small $\tau>0$, we have that 
$\hat J_{\mathcal V}=A_0- B_0D_0^{-1}C_0$  is invertible
since $\frac{\partial \bar{\rho}(t_{M+1})}{\partial (\bar{v}^1_T)}=\frac {\partial U}{\partial \mathcal V}.$

Therefore, we conclude that 
$J_{\tilde{\mathbf F}}$ is invertible in a sufficiently small neighborhood of $\bar{x}$. The remaining part of the argument, namely, the quadratic convergence of Newton's method, is standard.
\end{proof}

\subsection{Numerical integration error in Newton's method}

The following proposition provides a numerical error bound arising from the discretization in Newton's iteration and indicates that our numerical method is robust with respect to the choice of spanning tree.

\begin{proposition}[Numerical Error]
\label{prop:numerical error}
Let the conditions of Theorem \ref{main-thm} hold. 
Here, $\mathbf F$ denotes the nonlinear system obtained by approximating the integral terms in \eqref{eq:Ftilde-rho}--\eqref{eq:Ftilde-v} by the left-endpoint quadrature rule, whereas $\widetilde{\mathbf F}$ denotes the corresponding exact system. 

Assume that the initial vector $x^{(0)}$ for the Newton's method with $\mathbf F$ and $\widetilde {\mathbf F}$ is sufficiently close to $\bar{x}:=((\bar{\rho}_i(t_m))_{2\le m\le M,i\le N-1},(\bar{v}_f(t_m))_{m\le M+1,f\in E_T})^{\top}.$ 
Let  Newton's methods with $\mathbf F$ and $\widetilde {\mathbf F}$ 
satisfy a termination tolerance of $\epsilon$:
\[
\|\mathbf F(x^{(k)})\|<\epsilon
\quad\text{and}\quad
\|\widetilde{\mathbf F}(x^{(\widetilde k)})\|<\epsilon .
\]
Denote the corresponding  numerical solutions by  $x^{(k,\mathbf F)}$ and $ x^{(\widetilde k,\widetilde {\mathbf F})}$ with the stopping iteration being $k$ and $\widetilde k$, respectively.
Then it holds that  
\begin{equation}\label{numeric-time-err}
\|x^{(k,\mathbf{F})} - x^{(\widetilde{k},\widetilde{\mathbf{F}})}\|
\lesssim \tau + \epsilon \tau^{-1}.
\end{equation}
\end{proposition}

\begin{proof}
On the one hand, the termination tolerance yields that 
\begin{align*}
\mathbf F(x^{(k,\mathbf F)})=\mathcal O(\epsilon), \;{\widetilde {\mathbf F}}(x^{(\widetilde k,\widetilde {\mathbf F})}) =\mathcal O(\epsilon).
\end{align*}
This means that 
\begin{align*}
\rho_i^{m+1}&=
      \rho_i^m
    - \Phi_i \Big(t_m,t_{m+1},(\rho_i^m)_{i\le N-1},(v_f^m)_{f\in E_T}\Big)+\mathcal O(\epsilon), \; m\le M ,\\
v_{f}^{m+1}&=v_{f}^m 
    - \Phi_f \Big(t_m,t_{m+1},(\rho_i^m)_{i\le N-1},(v_f^m)_{f\in E_T}\Big)+\mathcal O(\epsilon),\; m\le M,
\end{align*}
and that
\begin{align*}
\widetilde \rho_i^{m+1}&=
      \widetilde \rho_i^m
    - \widetilde \Phi_i \Big(t_m,t_{m+1},(\widetilde \rho_i^m)_{i\le N-1},(\widetilde v_f^m)_{f\in E_T}\Big)+\mathcal O(\epsilon), \; m\le M ,\\
\widetilde v_{f}^{m+1}&=\widetilde v_{f}^m 
    - \widetilde \Phi_f \Big(t_m,t_{m+1},(\widetilde \rho_i^m)_{i\le N-1},(\widetilde v_f^m)_{f\in E_T}\Big)+\mathcal O(\epsilon),\; m\le M.
\end{align*}
Here, for simplicity, we omit the dependence of $\rho_i^{m},v_f^{m},\widetilde {\rho_i^{m}},\widetilde {v_f^{m}}$ on $k,\widetilde k$.

On the other hand, $\widetilde {\bf F}, \bf F$ are   continuously  differentiable w.r.t. $x$, and thus are both Lipschitz in a small neighborhood $\mathcal B(\bar{x},\eta)$ for some $\eta>0$ small enough.
Since $\mathbf F=\widetilde {\mathbf F}+\mathcal O(\tau^2)$  in this neighborhood, we further have that 
\begin{align*}
\|\rho^{m+1}-\widetilde{\rho}^{m+1}\|
+\|v^{m+1}-\widetilde{v}^{m+1}\|
&\lesssim  \|\rho^{m}-\widetilde{\rho}^{m}\|
+\|v^{m}-\widetilde{v}^{m}\|+\mathcal O(\tau^2)+\mathcal O(\epsilon).
\end{align*}
Gronwall's inequality yields that 
\begin{align*}
\|x^{(k,\mathbf F)}-x^{(\widetilde k,\widetilde {\mathbf F})}\|\lesssim \mathcal O(\tau)+\mathcal O(\epsilon\tau^{-1}).
\end{align*}
Hence the proof is complete.
\end{proof}
\begin{remark}
\label{the discrepancy of the spanning tree}
     Theorem \ref{main-thm} ensures the quadratic convergence of Newton's iteration, i.e., 
   \begin{align*}
   \|x^{(k,\widetilde {\mathbf F})}-\bar{x}\|\lesssim  \|x^{( k-1,\widetilde {\mathbf F})}-\bar{x}\|^2 \lesssim \eta^{2^{ k}}\le \epsilon,
   \end{align*}
   where $\eta$ is the initial error between $x^{(0)}$ and $\bar{x}$.
 Proposition~\ref{prop:numerical error}guarantees that the numerical integration error is $\mathcal O(\tau)$ due to the use of the left-rectangular rule.
   
Moreover, we note that the choice of the spanning tree does not affect the numerical computation; see Example~\ref{the independence of Tree}. Let $x^{(k_{1},\\T_1)}$ and $x^{(k_{2},T_2)}$ be two numerical solutions of Newton's method with the left-rectangular rule, where $T_i$, $i=1,2$, are two distinct spanning trees of the graph $G,$ and {$k_{i}$} is the corresponding stopping step. By Proposition~\ref{prop:numerical error}, and $\frac 1{\sqrt{\omega_{e}}}v_{e}=\sum_{f\in E_T}a_f \frac 1{\sqrt{\omega_{f}}} v_f, \; \text{for any}\; e\in E, f=(i,j)\in E_{TT_{n}}, n=1,2,$ and the mass conservation, the discrepancy between these two numerical solutions is
\(\mathcal O(\tau)+\mathcal O(\epsilon\tau^{-1})\).
\end{remark}

In the final part of this section, we address the preservation of positivity of the density, which plays a crucial role in the convergence of the Newton-type algorithm.

\subsection{Non-negativity preservation of the density}
\label{sec:positivity}

A key assumption in Theorem~\ref{main-thm} is that
the convergence of the Newton-type solver requires 
that the density remain positive at all time levels.
Since our time discretization replaces the exact temporal integration,
positivity is not automatically obtained and must be enforced at the level of the numerical flux.


We are now in a position to prove the non-negativity of the numerical
solution to the continuity equation arising from the discrete OT problem
when the probability weight is chosen to be the upwind weight. More precisely,
the following result concerns the same discrete OT formulation as before,
except that the average weight function is replaced by the upwind weight
function
\begin{equation}
\label{eq:theta-upwind_1}
\theta_{ij}^{\mathrm{up}}(\rho^m,v^m)=
\begin{cases}
\rho_i^m, & v_{i,j}^m \ge 0,\\[2pt]
\rho_j^m, & v_{i,j}^m < 0.
\end{cases}
\end{equation}
Under a suitable upwind CFL condition, the resulting fully discrete
continuity equation preserves the non-negativity of the density.

See Appendix~\ref{sec:invertibility-upwind} for the convergence of Newton's method applied to 
OT problems on graphs with an upwind probability weight function.

\begin{proposition}[Non-negativity preservation under an upwind CFL condition]
\label{prop:positivity}
Assume $\rho^0\ge 0$ componentwise and let $\theta_{ij}= \theta^{\mathrm{up}}_{ij}$ be defined by \eqref{eq:theta-upwind_1}.
Define $(v_{i,j}^m)^+ := \max(v_{i,j}^m,0)$ and $(v_{i,j}^m)^- := \max(-v_{i,j}^m,0)$ so that
$v_{i,j}^m=(v_{i,j}^m)^+-(v_{i,j}^m)^-$.
If, for all nodes $i$ and all time levels $m$, the time step $\tau$ satisfies
\begin{equation}
\label{eq:CFL-local}
\tau \sum_{j\in \mathcal N(i)} \sqrt{\omega_{ij}}\,(v_{i,j}^m)^+ \le 1,
\end{equation}
then the density update obtained from \eqref{eq:Ftilde-rho} is nonnegativity-preserving, that is
\[
\text{if}\quad \rho^m \ge 0, \quad\text{then}\quad  \rho^{m+1}\ge 0,
\qquad m=1,\dots,M.
\]
In particular, letting $d_{\max} := \max_{i\in V} |\mathcal N(i)|$,
$\omega_{\max}:=\max_{(i,j)\in E}\omega_{ij}$, and $\|v^m\|_{\infty}:=\max_{(i,j)\in E}|v_{i,j}^m|$,
a sufficient global condition is
\begin{equation}
\label{eq:CFL-global}
\tau \le \frac{1}{d_{\max} \,\sqrt{\omega_{\max}}\,\|v^m\|_{\infty}}
\qquad \text{for all } m.
\end{equation}
\end{proposition}

\begin{proof}
Inserting the upwind weight $\theta^{\mathrm{up}}_{ij}$ defined in \eqref{eq:theta-upwind_1} into \eqref{eq:Ftilde-rho}, we obtain
\begin{align*}
\rho_i^{m+1}
&= \rho_i^m
-\tau \sum_{j\in \mathcal N(i)} \sqrt{\omega_{ij}}
\Big((v_{i,j}^m)^+\,\rho_i^m - (v_{i,j}^m)^-\,\rho_j^m\Big)\\
&= \rho_i^m\Big(1-\tau \sum_{j\in \mathcal N(i)} \sqrt{\omega_{ij}}(v_{i,j}^m)^+\Big)
\;+\; \tau \sum_{j\in \mathcal N(i)} \sqrt{\omega_{ij}}(v_{i,j}^m)^-\,\rho_j^m.
\end{align*}
By the induction hypothesis, $\rho_j^m\ge 0$ for all $j$, and the second term is therefore nonnegative.
Under \eqref{eq:CFL-local}, the coefficient of $\rho_i^m$ is also nonnegative, hence $\rho_i^{m+1}\ge 0$.
Finally, \eqref{eq:CFL-global} follows from
\[\sum_{j\in\mathcal N(i)} \sqrt{\omega_{ij}}(v_{i,j}^m)^+
\le |\mathcal N(i)|\,\sqrt{\omega_{\max}}\,\|v^m\|_{\infty}\le d_{\max} \,\sqrt{\omega_{\max}}\,\|v^m\|_{\infty}.\]
\end{proof}




For numerical validation, we consider a benchmark experiment on a one dimensional periodic uniform lattice graph, together with a Gaussian mixture example on a grid; see Supplementary Example~SM0.1 and Example~SM0.3. 
To further evaluate the accuracy and convergence of the method, we also include benchmark tests in Supplementary Example~SM0.2, where the spatial and temporal convergence results are reported.

\section{Numerical examples}
\label{Numerical Examples}
In this section we report on numerical experiments for Algorithm~\ref{alg:OT-solver}. We begin with the qualitative behavior of the computed transport dynamics, including multimodal transport, the effect of spanning-tree parametrizations, and topology-induced patterns on non-lattice graphs; see Examples~\ref{the independence of Tree} and \ref{ex:dumbbell}. Then, we present extensions of our basic setup on important applications, such as inverse topology inference and opinion-dynamics modeling; see Examples~\ref{Topology recovery from flux sparsification} and \ref{ex:social-relax}. We conclude with a limited comparison of our approach to other approaches; see Section \ref{comparison}.
In the benchmark tests, the Jacobian is assembled analytically, whereas in the remaining experiments it is approximated by a first-order divided-difference formula. Periodic boundary conditions are imposed in all lattice-based simulations. The following Figure~\ref{the graph struture for examples} depicts the structural properties of several graphs.
\begin{figure}[ht]
    \centering
    \includegraphics[trim=0cm 0cm 0cm 1cm, clip, width=0.9\textwidth]{ 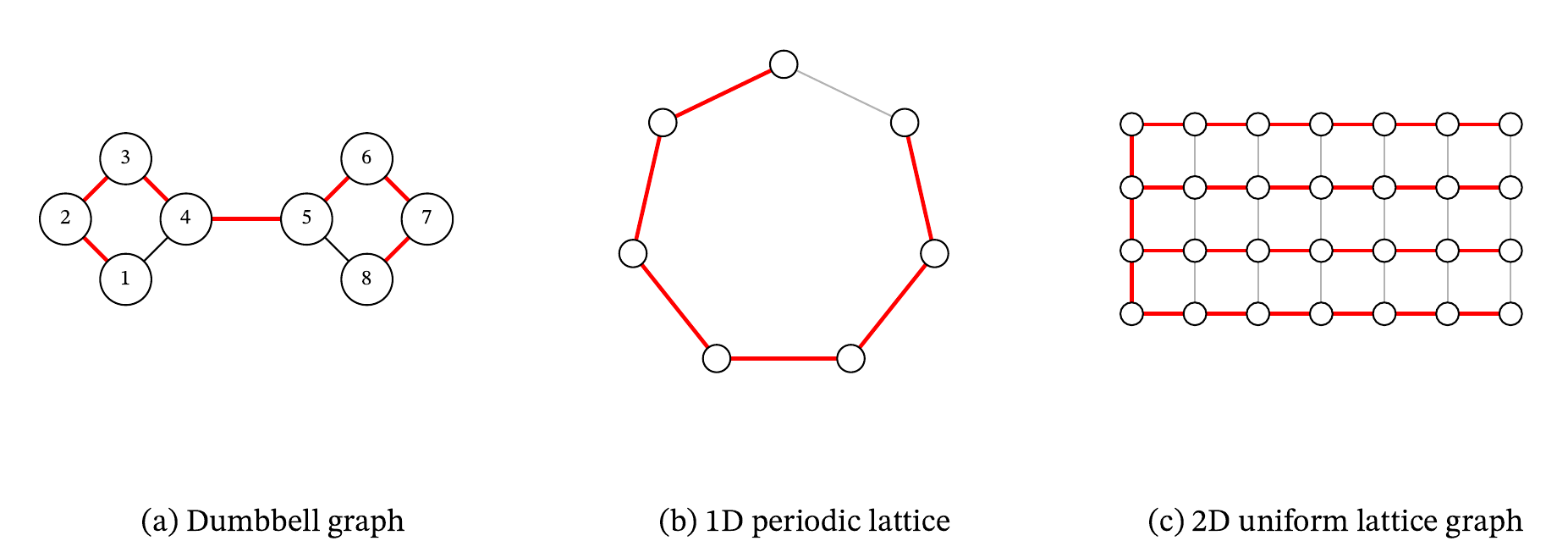}
    \caption{Panel (a) illustrates the graph structure in Example~\ref{ex:dumbbell}, whereas panels (b) and (c) present the spanning trees employed in Comparison Test~\ref{comparison}. The highlighted edges represent the corresponding spanning trees.}
    \label{the graph struture for examples}
\end{figure}

\subsection{Algorithmic properties}
\subsubsection{Consistency of the solutions under different spanning trees}
\noindent
Below we show numerically that different spanning-tree choices do not lead to substantially different solutions.

Figure~\ref{W2-MST-comparison}(a) reports the discrete squared Wasserstein distance $W_2^2$ computed on the same underlying graph with different spanning-tree parametrizations. Within the discrete dynamical
formulation, $W_2^2$ can be estimated by two quantities that coincide in the structure-preserving (exact)
setting and remain close under time discretization and inexact Newton solves. The first estimator is the
time-integrated discrete action,
\begin{equation}\label{eq:W2_action}
    W_2^2(\rho^0,\rho^1)
    \;\approx\;
    a := \sum_{n=0}^{M-1} \tau
    \sum_{(i,j)\in E}
    \theta_{ij}(\rho^n)\, |v_{i,j}^n|^2 ,
\end{equation}
where $\tau$ is the time step size, $M$ is the number of time steps, $v_{i,j}^n$ denotes the discrete edge
velocity at time level $t_n=n\tau$, and $\theta_{ij}(\rho^n)$ is the density-dependent edge weight at time $t_n$. An
alternative estimator exploits the Hamiltonian structure of the discretized Benamou--Brenier system and
evaluates the initial kinetic energy,
\begin{equation}\label{eq:W2_initial}
    W_2^2(\rho^0,\rho^1)
    \;\approx\;
    b := \sum_{(i,j)\in E}
    \theta_{ij}(\rho^0)\, |v_{i,j}^0|^2 ,
\end{equation}
where $v_{i,j}^0$ is the discrete initial velocity recovered from the Newton solution of the discretized optimality system.

\begin{example}[Independence from spanning-tree parametrization]
\label{the independence of Tree}
The experiment is conducted on the five-node graph shown in Figure~\ref{W2-MST-comparison}(b), consisting of a
cycle augmented by the diagonal edge $(1,3)$. We consider three spanning trees:
(i) $\texttt{T}_1$ with edges $(1,2),(2,3),(3,4),(4,5)$;
(ii) $\texttt{T}_2$ with edges $(2,3),(3,4),(4,5),(1,5)$; and
(iii) $\texttt{T}_3$ with edges $(2,3),(1,3),(1,5),(4,5)$.
Although these choices induce different parametrizations of the velocity variables along the tree edges,
the resulting values of $W_2^2$ computed from both estimators $a$ and $b$ are nearly identical for all time
resolutions $M$.

The estimated distances are not strictly identical: the small discrepancies are attributable to (i) the temporal
discretization error and (ii) the finite termination tolerance of the Newton iteration (and, to a lesser
extent, floating-point round-off). This behavior is consistent with the stability statement in
Proposition~\ref{prop:numerical error}, which implies that the computed transport cost is independent of the
spanning tree choice up to numerical errors. Consequently, the variations observed in
Figure~\ref{W2-MST-comparison}(a) should be interpreted as numerical effects rather than artifacts of a
particular spanning-tree parametrization; the underlying Wasserstein distance remains an intrinsic geometric quantity
associated with the graph.
\begin{figure}[ht]
\centering

\begin{tabular}{@{}c@{\hspace{0.04\textwidth}}c@{}}
\includegraphics[width=0.48\textwidth]{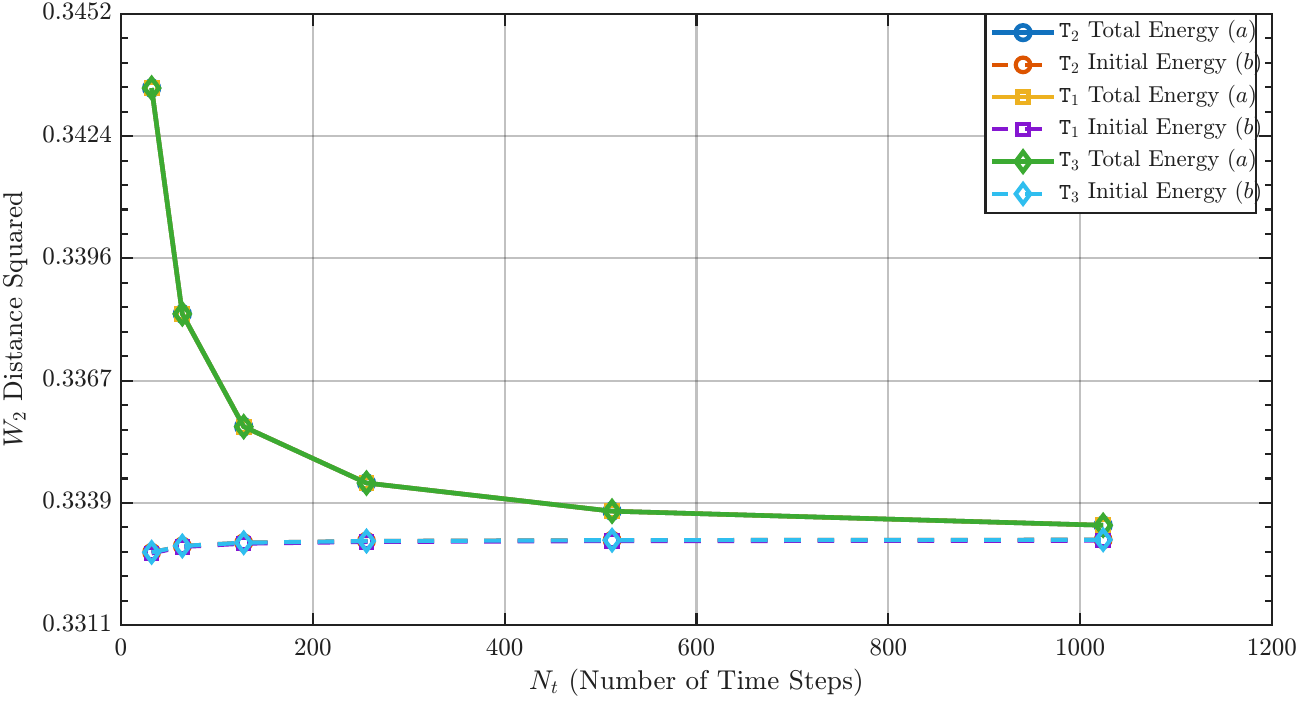}
&
\begin{tikzpicture}[scale=0.8,
    vertex/.style={circle, draw, minimum size=7mm, inner sep=0pt}]
    \node[vertex] (1) at (90:2) {1};
    \node[vertex] (5) at (18:2) {5};
    \node[vertex] (4) at (-54:2) {4};
    \node[vertex] (3) at (-126:2) {3};
    \node[vertex] (2) at (162:2) {2};

    \draw (1)--(2);
    \draw (2)--(3);
    \draw (3)--(4);
    \draw (4)--(5);
    \draw (5)--(1);
    \draw (1)--(3);
\end{tikzpicture}
\\[-0.2ex]
\parbox[t][2.6ex][t]{0.48\textwidth}{\centering {\footnotesize(a)  Comparison under different spanning trees}}
&
\parbox[t][2.6ex][t]{0.48\textwidth}{\centering {\footnotesize (b) Underlying graph used in the experiment}} 
\end{tabular}

\caption{Example~\ref{the independence of Tree}.  Comparison of the discrete $W_2^2$ distance under different spanning
spanning-tree parametrizations and the corresponding five-node graph structure.}
\label{W2-MST-comparison}
\end{figure}
\end{example}

\subsubsection{Topology-induced transport patterns: a dumbbell graph}
\noindent

Example~\ref{the independence of Tree} highlighted that the cost is not impacted by the choice of a tree.  Below, 
instead, we highlight that the \emph{transport dynamics} are strongly shaped by the connectivity of the underlying graph.

\begin{example}[Dumbbell graph: connectivity-induced transport bottleneck]
\label{ex:dumbbell}
We consider a \emph{dumbbell graph} $G=(V,E,\Omega)$ formed by connecting two densely connected subgraphs
(a \emph{left} cluster and a \emph{right} cluster) through a small set of \emph{bridge edges}, as shown in Figure~\ref{the graph struture for examples}(a).
In our construction, the two clusters are linked by the single bridge edge $(4,5)$.
The edge weights $\Omega=(\omega_{ij})_{(i,j)\in E}$ are chosen 
 to be uniform, with $\omega_{ij}\equiv 1$. The initial and terminal densities are generated randomly and then normalized to unit mass.

\begin{figure}[htbp]
    \centering
    \subfloat[Evolution of nodal densities $\rho_i(t)$.\label{fig:density}]{
        \includegraphics[width=0.435\textwidth]{ 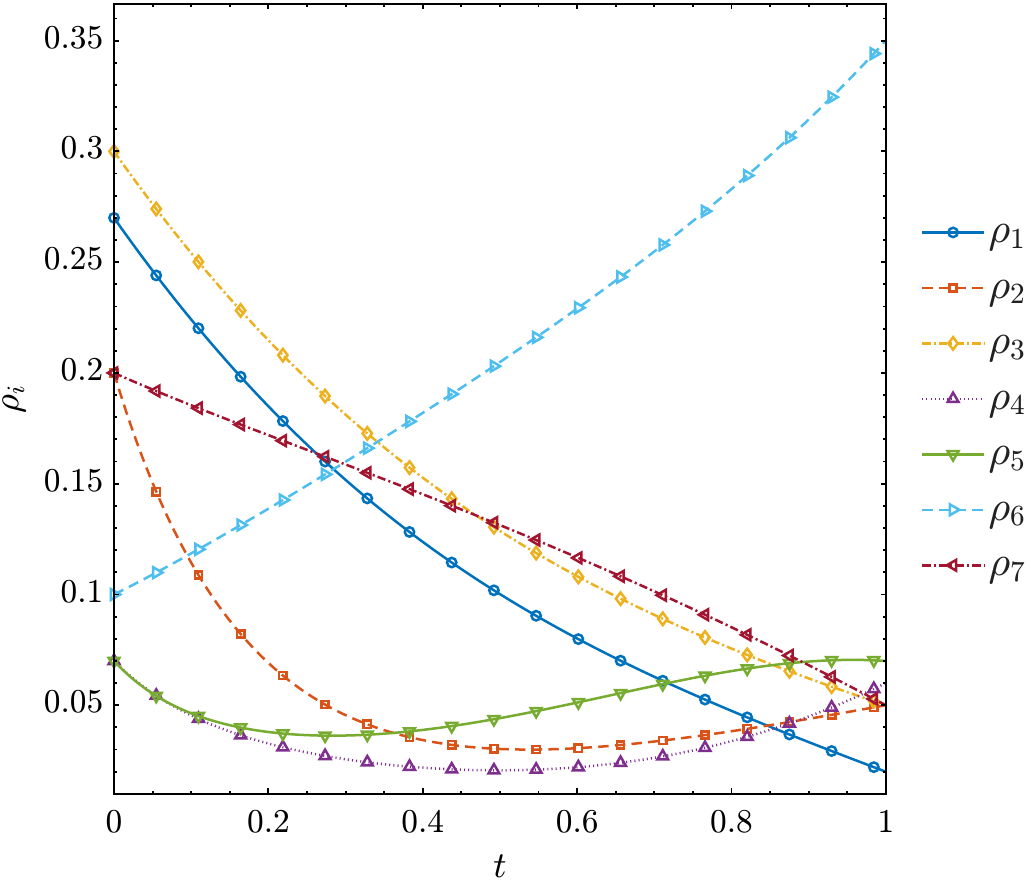}
    }
    \hfill
    \subfloat[Evolution of edge velocities $v_{i,j}(t)$.\label{fig:velocity}]{
        \includegraphics[width=0.426\textwidth]{ 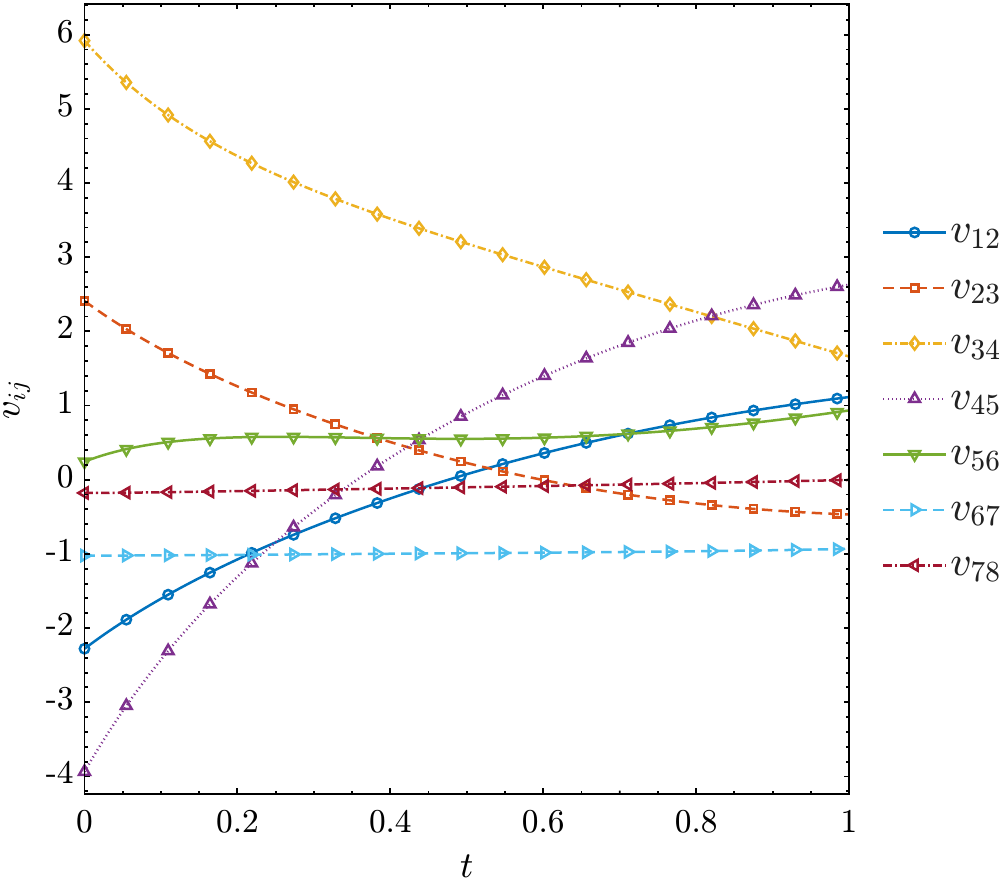}
    }
    \caption{Example~\ref{ex:dumbbell}. Dumbbell-graph experiment with \(M=128\) and \(|V|=N=8\).}
    \label{fig:numerical_results}
\end{figure}
\end{example}

Figure~\ref{fig:numerical_results} reports the numerical solution on the dumbbell graph with temporal discretization parameter $M = 128$.
As expected, the bridge edge \((4,5)\) acts as a bottleneck and carries most of the inter-cluster flux, since the velocity along this edge remains relatively large throughout the transport process.
This concentration of flow typically results in larger temporal variability in the corresponding edge velocity. 
The nodal densities $\rho_i(t)$ remain strictly
positive throughout the simulation, indicating that the scheme produces a stable and physically
consistent transport trajectory. Overall, the experiment highlights that the \emph{transport pattern} is
highly sensitive to the graph connectivity: bottleneck edges concentrate cross-cluster transport and
amplify velocity fluctuations, while the densities remain regular.

\subsection{Advanced applications}
The experiments above establish two key points: (i) the computed transport cost is robust with respect to the choice of spanning tree,
and (ii) the structure of the graph can strongly influence the resulting transport dynamics, as illustrated by the bottleneck behavior on the dumbbell graph. 

Building on these algorithmic and modeling properties, we now demonstrate how the proposed dynamical OT framework can serve as a versatile tool beyond the forward computation of a single transport distance. In particular, we present two applications: (1) an inverse task in which flux patterns are used to prune edges and infer an effective network topology, and (2) an interpretation of the same OT dynamics as an interaction model for opinion evolution on social networks.

\subsubsection{Inverse OT for time-dependent topology inference}

In this section, we study an inverse problem motivated by information diffusion on social networks. Given the transport dynamics induced by the dynamical OT model, our objective is to infer a sparse, time-varying effective communication topology that governs the redistribution process, based on the observed initial and final states.

We begin from an overparameterized hypothesis graph, taken to be the complete graph $K_{|V|}$, so that every pair of nodes is initially allowed to exchange information. 
This choice reflects the absence of prior knowledge about the actual transmission channels: any two users may potentially communicate, but only a subset of these candidate edges will be effectively activated by the transport dynamics.

For a message spreading over a social network, we interpret the node variable $\rho_i(t)$  as the normalized message-related activity or influence intensity of that user at time $t$. 
Accordingly, the initial and terminal distributions $\mu$ and $\nu$ describe the allocation of message-related activity over the network at the beginning and at the end of the diffusion process. 
Under this interpretation, a decrease in $\rho_i(t)$ indicates that the relative participation of that node in the global diffusion process has decreased.

Solving the dynamical OT on the hypothesis graph yields a time-dependent flux field $v_{i,j}(t)$, which describes how message-related activity is transferred across the network in a cost-efficient manner. 
These fluxes can be interpreted as revealing latent communication pathways induced by the observed transition from $\mu$ to $\nu$.

To extract an effective topology from the dense hypothesis graph, we apply a thresholding rule to the edge fluxes at each time step. 
Given a tolerance parameter $\varepsilon>0$, we define the time-dependent effective edge set by
\begin{equation}\label{eq:effective_edge_set_time}
(i,j)\in E_{\mathrm{eff}}(t)
\quad\Longleftrightarrow\quad
|v_{i,j}(t)|>\varepsilon.
\end{equation}
Edges whose flux magnitude does not exceed the threshold at time $t$ are regarded as inactive at that instant. 
In this way, the complete graph is reduced, at each time step, to a sparse effective communication network supported by the OT dynamics.

Therefore, the inverse problem considered here is  to infer a time-dependent effective transmission structure that explains how the observed message-related activity is redistributed across the network.

We next present a numerical experiment to illustrate the behavior of this inverse-recovery heuristic.

\begin{example}[Topology recovery from flux sparsification]
\label{Topology recovery from flux sparsification}
We solve the dynamical OT problem on the complete graph $K_{|V|}$.  
Although the optimization is performed on a fully connected graph, a spanning tree is selected beforehand for gauge fixing; this backbone is shown in Figure~\ref{fig:evolution_edge}(a).

After computing the edge velocities $v_{i,j}(t)$, we then declare an edge $(i,j)$ to be active if this quantity exceeds a prescribed threshold \eqref{eq:effective_edge_set_time}.  
The resulting inferred connectivity is shown in Figure~\ref{fig:evolution_edge}(b), where edge colors indicate the value of $\|{v}_{i,j}\| $.

Figure~\ref{fig:evolution_edge} illustrates the topology recovery procedure used to recover the underlying connectivity of the graph. The figure as a whole shows how the graph’s effective topology evolves over time.
Panel (a) shows the spanning tree selected from the complete graph, which is used as the backbone for gauge fixing.
For each edge $(i,j)$, we compute the edge velocity $v_{i,j}$. Panel~(b) shows the intermediate connectivity state, while panel~(c) displays the effective connectivity obtained by retaining only those edges whose the norm of velocity exceeds the prescribed threshold. The edge colors indicate the magnitude of $v_{i,j}$.
\end{example}

\begin{figure}[ht]
    \centering
    \includegraphics[width=0.7\textwidth]{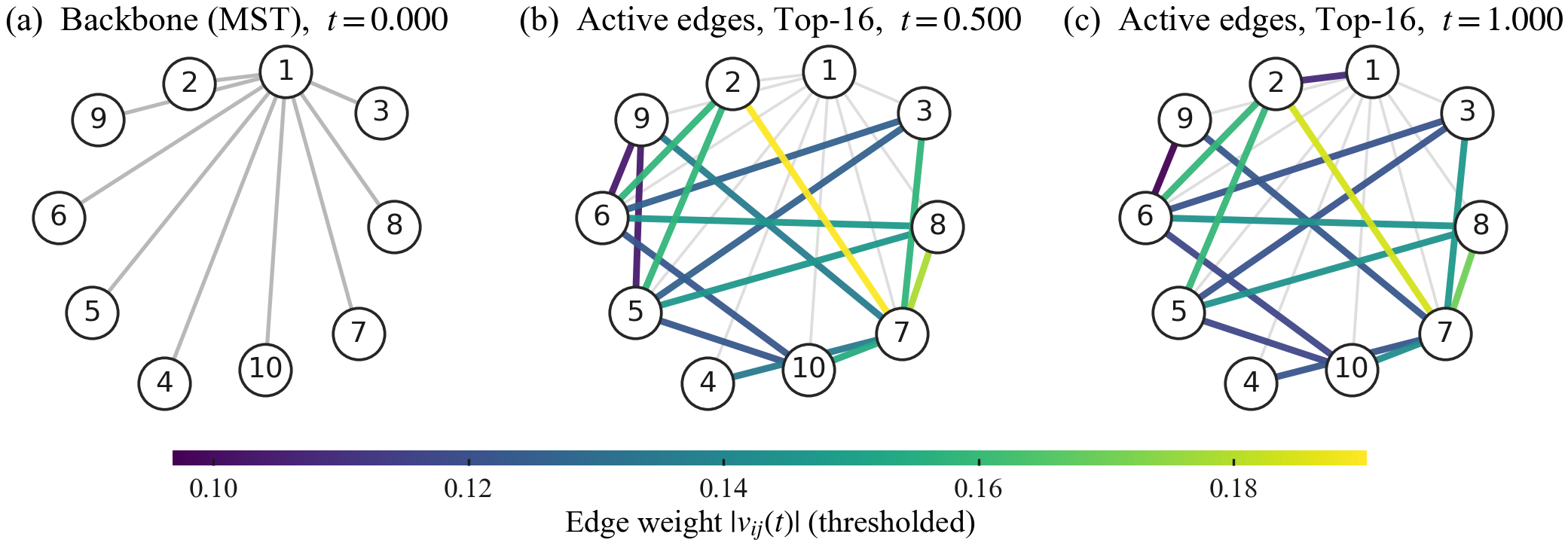}
    \caption{Example \ref{Topology recovery from flux sparsification}.  Topology recovery from the average edge velocity with temporal
discretization parameter $M = 128$.}
    \label{fig:evolution_edge}
\end{figure}

\subsubsection{Opinion dynamics on social networks}
In the previous sections, we assumed that the underlying graph $G=(V,E,\Omega)$ is given and
formulated the dynamical OT problem directly on $G$.
Here we adopt an equivalent but more modeling-oriented viewpoint: the network connectivity
is specified by an adjacency matrix, so that the same graph-based OT dynamics can be read
as an interaction model on a social network.

In this interpretation, each vertex $i\in V$ represents an agent, and the nodal variable
$\rho_i(t)\ge 0$ denotes the agent's normalized belief (or support) level for a fixed opinion.
We enforce the normalization $\sum_{i\in V}\rho_i(t)=1$; hence $\rho_i(t)\in[0,1]$, and in
particular $\rho_i(t)=1/|V|$ corresponds to the uniform consensus state.
The edge field $v_{i,j}(t)$ describes the interaction-driven redistribution of belief mass
along the network links, so the adjacency structure determines which agents can directly
influence each other.

\begin{example}[Relaxation to consensus on a small network]
\label{ex:social-relax}
We examine the time evolution of $\rho(t)$ on a network with $|V| = N = 10$ nodes, defined by a binary adjacency matrix $A \in \{0,1\}^{N \times N}$. The corresponding spanning tree is constructed using Kruskal’s algorithm. The initial distribution $\mu$ is drawn randomly (seen in
Figure~\ref{general}), while the target distribution is the uniform state
$\nu \equiv 1/N$, representing a consensus configuration.
The dynamics are governed by~\eqref{equation expression1}. 
We solve the resulting system using Algorithm~\ref{alg:OT-solver} with temporal discretization parameter $M = 256$, and the corresponding numerical results are presented in Figure~\ref{general}. Figure~\ref{general} illustrates a relaxation-to-consensus process: starting from a heterogeneous initial state, different users exhibit varying degrees of support or opposition (as encoded by the nodal density values \(\rho_i\)). As time evolves, these differences are gradually smoothed out and the system converges to a consensus configuration, represented by the uniform distribution \(\nu_i \equiv 1/|V|\). In this sense, \(\rho(t)\) directly visualizes the evolution of users' opinions across the network.

At the same time, the edge velocity field \(v(t)\) offers a complementary perspective on how consensus is formed. The magnitude of \(v_{i,j}\) indicates the strength of interaction or influence transmitted along the edge \((i,j)\): larger values highlight the neighboring users that most strongly affect the evolution of a given node's state. Together, the coupled dynamics of \(\rho\) and \(v\) provide an intuitive and quantitative description of opinion evolution on social media, capturing both the macroscopic trend toward consensus and the underlying influence pathways that drive this evolution.

\begin{figure}[ht]
    \centering
    \includegraphics[width=0.8\textwidth]{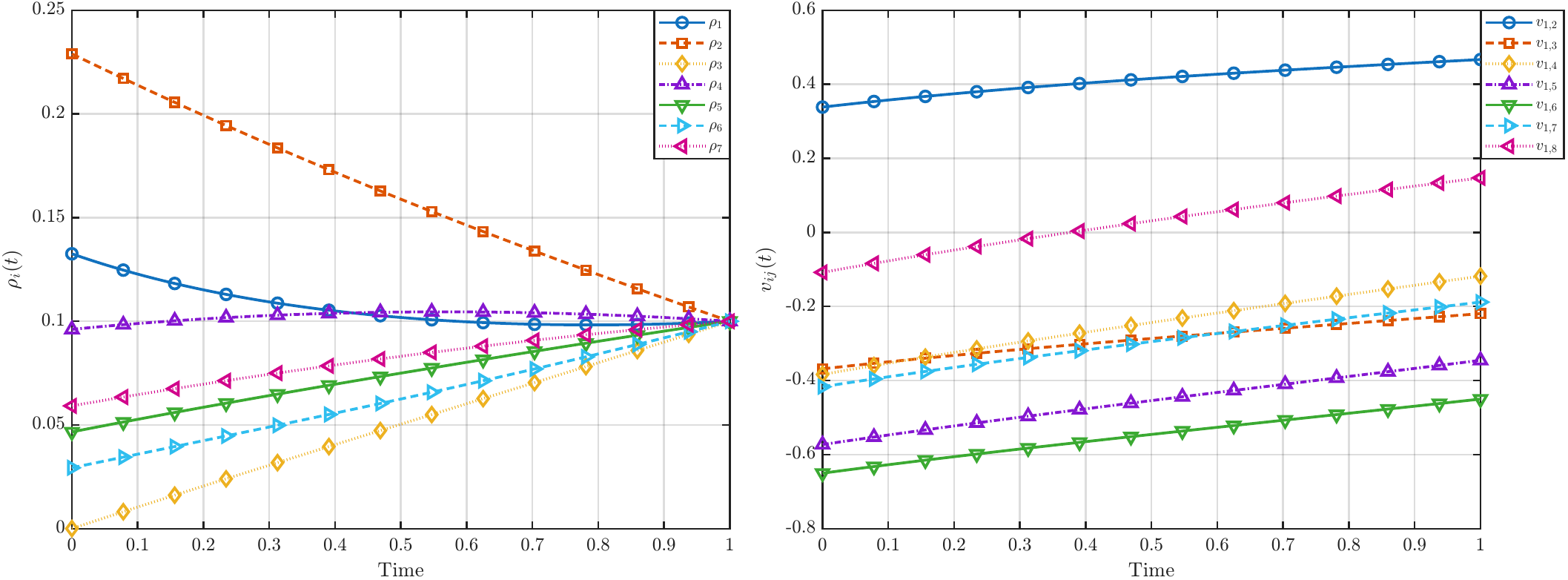}
    \caption{Evolution of the nodal density $\rho$ and edge velocity $v$ from a random initial state to the uniform target state in Example~\ref{ex:social-relax}, with $M = 256$ time steps.}
    \label{general}
\end{figure}
\end{example}

\subsection{A comparison test}\label{comparison}
To assess computational efficiency, we compare the wall-clock runtime of the multiple shooting method and the current Newton-based method under identical experimental settings. All experiments were performed in \textsc{Matlab} R2023b on macOS using a MacBook Air equipped with an Apple M3 chip (8-core CPU and 10-core GPU).

 For the comparison on the uniform lattice graph \(G=(V,E,\Omega)\) (see Fig.~\ref{the graph struture for examples}(b)), viewed as a discretization of the spatial domain \(\mathcal{O}=[-1,3]\), we take \(|V|=64\) and assign the nodal weight
$\Omega=\{\omega_{i,i+1}\}_{1 \leq i \leq N} = \frac{1}{16^2}$. 
Periodicity is enforced by identifying $x_{N}\equiv x_0$ and including the wrap-around edge $(N-1,0)\in E$. The initial and terminal densities are chosen as Gaussian-type profiles
\begin{equation}\label{eq:1d_gaussian}
\hat{{\mu}}(x)=\exp\!\left(-a_{\mu}(x-b_{\mu})^2\right)+r_{\mu},\qquad
\hat{{\nu}}(x)=\exp\!\left(-a_{\nu}(x-b_{\nu})^2\right)+r_{\nu},
\end{equation} with parameters
$a_{\mu}=a_{\nu}=15$, $b_{\mu}=1.4$, $b_{\nu}=1.7$, and $r_{\mu}=r_{\nu}=10^{-4}$.

For the two-dimensional experiment, we consider the uniform lattice graph
\(G=(V,E,\Omega)\) (see Fig.~\ref{the graph struture for examples}(c)) associated with the spatial domain \(\mathcal{O}=[-1,3]^2\),
with \(|V|=64\times 64\). The edge set \(E\) consists of all horizontal and
vertical nearest-neighbor pairs, and we assign the uniform edge weight
\(\Omega=\{\omega_{ij}\}_{(i,j)\in E}=\frac{1}{16^2}\). The initial and terminal
densities are chosen as Gaussian-type profiles
\begin{equation}
\label{eq:mu_mixture}
\begin{split}
\hat{\mu}(x)
=\, & w_{\mu}\exp\!\left(
-a_{\mu}(x_1-b_{\mu})^2
-c_{\mu}(x_2-d_{\mu})^2
\right) + \varepsilon,
\end{split}
\end{equation}
and
\begin{equation}
\label{eq:nu_mixture}
\begin{split}
\hat{\nu}(x)
=\, & \exp\!\left(
-a_{\nu}(x_1-b_{\nu})^2
-c_{\nu}(x_2-d_{\nu})^2
\right) + \varepsilon.
\end{split}
\end{equation}
The parameters
$a_{\mu}$, $c_{\mu}$, $a_{\nu}$ and $c_{\nu} > 0$
determine the concentration of the Gaussian components in the $x_1$- and $x_2$-directions, respectively, while
$\bigl(b_{\mu}, d_{\mu}\bigr)$ and
$\bigl(b_{\nu}, d_{\nu}\bigr)$ denote the corresponding centers in the initial and terminal profiles. In this example, we choose $a_{\mu}=a_{\nu}=10$, $c_{\mu}=c_{\nu}=10$,
$(b_{\mu}, d_{\mu}) = (0.5,1.5)$,  $(b_{\nu}, d_{\nu}) = (1.5,1.3)$ and $\varepsilon =0.0001 $. Table \ref{tab:cpu_1d2d} shows the CPU times of 
multiple shooting(M.S.) and Newton's method(N.M.) on the graph.

\begin{table}[htbp]
\centering
\caption{CPU times for 1D and 2D experiments for different time steps $M$.}
\label{tab:cpu_1d2d}
\begin{tabular*}{\textwidth}{@{\extracolsep{\fill}} l
                S[table-format=3.2]
                S[table-format=3.2]
                S[table-format=3.2]
                S[table-format=3.2] @{}}
\toprule
& \multicolumn{2}{c}{\textbf{1D}} & \multicolumn{2}{c}{\textbf{2D}} \\
\cmidrule(lr){2-3}\cmidrule(lr){4-5}
\textbf{Time Step} 
& {\textbf{M.S.}} & {\textbf{N.M.}}
& {\textbf{M.S.}} & {\textbf{N.M.}} \\
\midrule
$M = 16$  & 2.73 & 0.23 & 333.68 & 76.58 \\
$M = 32$  & 21.98  & 7.63   & 357.85 & 146.34 \\
$M = 64$  & 92.78  & 28.12  & 498.16 & 411.05 \\
$M = 128$ & 368.61 & 238.91 & 913.24 & 823.43 \\
\bottomrule
                                                                                                                                                                                                                                                                                                                                                                                                                                                                                                                                                                                                                                                                                                                                                                                                                                                                                                                                                                                                                                                                                                                                                                                                                                                                                                                                                                                                                                                                                                                                                                                                                                                                                                                                                                                                                                                                                                                                                                                                                                                                                                                                                                                                                                                                                                                                                                                                                                                                                                                                                                                                                                                                                                                                                                                                                                                                                                                                                                                                                                                                                                                                                                                                                                                                                                                                                                                                                                                                                                                                                                                                                                                                                                                                                                                                                                                                                                                                                                                                                                                                                                                                                                                                                                                                                                                                                                                                                                                                                                                                                                                                                                                                                                                                                                                                                                                                                                                                                                                                                                                                                                                                                                                                                                                                                                                                                                                                                                                                                                                                                                                                                                                                                                                                                                                                                                                                                                                                                                                                                                                                                                                                                                                                                                                                                                                                                                                                                                                                                                                                                                                                                                                                                                                                                                                                                                                                                                                                                                                                                                                                                                                                                                                                                                                                                                                                                                                                                                                                                                                                                                                                                                                                                                                                                                                                                                                                                                                                                                                                                                                                                                                                                                                                                                                                                                                                                                                                                                                                                                                                                                                                                                                                                                                                                                                                                                                                                                                                                                                                                                                                                                                                                                                                                                                                                                                                                                                                                                                                                                                                                                                                                                                                                                                                                                                                                                                                                                                                                                                                                                                                                                                                                                                                                                                                                                                                                                                                                                                                                                                                                                                                                                                                                                                                                                                                                                                                                                                                                                                                                                                                                                                                                                                                                                                                                                                                                                                                                                                                                                                                                                                                                                                                                                                                                                                                                                                                                                                                                                                                                                                                                                                                                                                                                                                                                                                                                                                                                                                                                                                                                                                                                                                                                                                                                                                                                                                                                                                                                                                                                                                                                                                                                                                                                                                                                                                                                                                                                                                                                                                                                                                                                                                                                                                                                                                                                                                                                                                                                                                                                                                                                                                                                                                                                                                                                                                                                                                                                                                                                                                                                                                                                                                                                                                                                                                                                                                                                                                                                                                                                                                                                                                                                                                                                                                                                                                                                                                                                                                                                                                                                                                                                                                                                                                                                                                                                                                                                                                                                                                                                                                                                                                                                                                                                                                                                                                                                                                                                                                                                                                                                                                                                                                                                                                                                                                                                                                                                                                                                                                                                                                                                                                                                                                                                                                                                                                                                                                                                                                                                                                                                                                                                                                                                                                                                                                                                                                                                                                                                                                                                                                                                                                                                                                                                                                                                                                                                                                                                                                                                                                                                                                                                                                                                                                                                                                                                                                                                                                                                                                                                                                                                                                                                                                                                                                                                                                                                                                                                                                                                                                                                                                                                                                                                                                                                                                                                                                                                                                                                                                                                                                                                                                                                                                                                                                                                                                                                                                                                                                                                                                                                                                                                                                                                                                                                                                                                                                                                                                                                                                                                                                                                                                                                                                                                                                                                                                                                                                                                                                                                                                                                                                                                                                                                                                                                                                                                                                                                                                                                                                                                                                                                                                                                                                                                                                                                                                                                                                                                                                                                                                                                                                                                                                                                                                                                                                                                                                                                                                                                                                                                                                                                                                                                                                                                                                                                                                                                                                                                                                                                                                                                                                                                                                                                                                                                                                                                                                                                                                                                                                                                                                                                                                                                                                                                                                                                                                                                                                                                                                                                                                                                                                                                                                                                                                                                                                                                                                                                                                                                                                                                                                                                                                                                                                                                                                                                                                                                                                                                                                                                                                                                                                                                                                                                                                                                                                                                                                                                                                                                                                                                                                                                                                                                                                                                                                                                                                                                                                                                                                                                                                                                                                                                                                                                                                                                                                                                                                                                                                                                                                                                                                                                                                                                                                                                                                                                                                                                                                                                                                                                                                                                                                                                                                                                                                                                                                                                                                                                                                                                                                                                                                                                                                                                                                                                                                                                                                                                                                                                                                                                                                                                                                                                                                                                                                                                                                                                                                                                                                                                                                                                                                                                                                                                                                                                                                                                                                                                                                                                                                                                                                                                                                                                                                                                                                                                                                                                                                                                                                                                                                                                                                                                                                                                                                                                                                                                                                                                                                                                                                                                                                                                                                                                                                                                                                                                                                                                                                                                                                                                                                                                                                                                                                                                                                                                                                                                                                                                                                                                                                                                                                                                                                                                                                                                                                                                                                                                                                                                                                                                                                                                                                                                                                                                                                                                                                                                                                                                                                                                                                                                                                                                                                                                                                                                                                                                                                                                                                                                                                                                                                                                                                                                                                                                                                                                                                                                                                                                                                                                                                                                                                                                                                                                                                                                                                                                                                                                                                                                                                                                            \end{tabular*}
\end{table}

\subsection{Comparison with other approaches}
It is well understood that there are important situations where the optimal transport of the initial density into the final one is achieved by a map, which is itself the gradient of a smooth function solving the Monge-Amp\'ere equation; e.g., see \cite{santambrogio2015optimal}.
For this reason, several works have tackled directly this map and solve the Monge--Amp\`ere equation; 
e.g., see 
\cite{benamou1999numerical,benamou2014numerical,froese2012transport,Oliker1989,Prins2015,WELLER2016102}.

Motivated by the work \cite{DieciOmarov2023}, we also perform a limited comparison of our approach to direct solution of the Monge-Amp\'ere equation on the model problem below.

\begin{example}\label{MAvsBB}
 We consider the same one-dimensional benchmark 
with initial density $\rho_0(x)=1+\sin(2\pi x)/32$ and target density $\rho_1(x)=1$
defined on a uniform lattice over $\mathcal{O}=[0,1]$.
As already observed in \cite{DieciOmarov2023}, 
directly solving the one-dimension Monge--Amp\`ere equation with a straightforwards 2nd order Runge-Kutta scheme, RK-2, is less expensive than other options.  The drawback of course is that by seeking the map we have a static formulation, and not automatically have the time-evolution for the density and velocity field. As for the dynamic formulations, finite difference method with Newton's iteration and multiple shooting method directly solve the geodesic equation unlike the regularized Benamou-Brenier methods \cite{LIwuchenBB}. And we can see that our finite difference method with Newton's iteration is the more efficient of these three options, see Table \ref{tab:comparison}.\\
\paragraph{Map error}
In Table \ref{tab:comparison}, we also show the error in the computed map for this example.  In fact, in this case the optimal transport map is given by
\begin{equation}
    T^\ast(x) = \left( x + \frac{1}{64\pi} \cos(2\pi x) \right) \pmod 1.
\end{equation}
Moreover, it is also known that the optimal map is given by the identity map plus the optimal initial velocity obtained from the Benamou-Brenier formulation.  Hence, we can compare methods based on directly solving the Monge-Amp\'ere problem to those who adopt the dynamic formulation of Benamou-Brenier.  Again, we report in Table \ref{tab:comparison} on the results of our comparison.
%
%
%
%
\begin{table}[htbp]
\centering
\caption{Performance comparison across grid refinements and solution methods for Example~\ref{MAvsBB}}
\label{tab:comparison}
\small
\begin{tabularx}{\textwidth}{@{}lcccc@{}}
\toprule
\textbf{Method} & $\boldsymbol{\mathrm{d}x}$ & $\boldsymbol{W_2(\rho_0,\rho_1)}$ & \textbf{Map Error} & \textbf{CPU Time} \\
\midrule
1-d Monge--Amp\`ere: & $1/16$  & $3.62\times10^{-3}$ & $1.28\times10^{-4}$ & $1.28\times10^{-3}$ \\
Shooting method with RK-2 & $1/32$  & $3.57\times10^{-3}$ & $3.20\times10^{-4}$ & $1.62\times10^{-3}$ \\
 & $1/64$  & $3.54\times10^{-3}$ & $7.99\times10^{-6}$ & $1.41\times10^{-3}$ \\
 & $1/128$ & $3.53\times10^{-3}$ & $2.00\times10^{-6}$ & $1.76\times10^{-3}$ \\
\midrule
Multiple Shooting & $1/16$  & $3.54\times10^{-3}$ & $9.91\times10^{-4}$ & $1.29$ \\
$(\mathrm{d}t=1/64)$ & $1/32$  & $3.52\times10^{-3}$ & $4.97\times10^{-4}$ & $2.26$ \\
 & $1/64$  & $3.52\times10^{-3}$ & $2.57\times10^{-4}$ & $10.03$ \\
 & $1/128$ & $3.52\times10^{-3}$ & $1.41\times10^{-3}$ & $35.28$ \\
\midrule
Regularized Benamou-Brenier & $1/16$  & $3.43\times10^{-3}$ & $3.10\times10^{-4}$ & $3.44$ \\
$(\mathrm{d}t = 1/21)$ & $1/32$  & $3.47\times10^{-3}$ & $2.46\times10^{-4}$ & $18.87$ \\
$\beta = 10^{-5}$, $\alpha = 0.5$ & $1/64$  & $3.49\times10^{-3}$ & $2.18\times10^{-4}$ & $101.05$ \\
 & $1/128$ & $3.50\times10^{-3}$ & $1.28\times10^{-4}$ & $854.65$ \\
\midrule
Finite Difference Method: & $1/16$  & $3.54\times10^{-3}$ & $9.78\times10^{-4}$ & $1.88\times10^{-1}$ \\
with Newton's iteration & $1/32$  & $3.52\times10^{-3}$ & $4.89\times10^{-4}$ & $1.06\times10^{-1}$ \\
$(\mathrm{d}t = 1/64)$ & $1/64$  & $3.52\times10^{-3}$ & $2.44\times10^{-4}$ & $1.88\times10^{-1}$ \\
 & $1/128$ & $3.52\times10^{-3}$ & $1.22\times10^{-4}$ & $5.50\times10^{-1}$ \\
\bottomrule
\end{tabularx}
\end{table}
As expected, from the numerical results reported in Table~\ref{tab:comparison}, we observe that
the one-dimensional Monge--Amp\`ere approach combined with the RK-2 shooting method
achieves the highest accuracy and the lowest computational cost among all methods tested, insofar as finding the map, but of course provide no information on the density and velocity evolutions.
By contrast, the finite difference method with Newton's iteration and the multiple shooting
method \cite{doi:10.1137/21M142160X} are dynamic approaches that solve the Wasserstein
geodesic equations directly and thus recover the full transport dynamics.
Compared with regularized Benamou--Brenier formulations, these methods avoid entropic
regularization and preserve the underlying geodesic structure of the problem.  Among the dynamic methods considered here, our finite difference technique coupled with Newton's method
is clearly less costly, while achieving comparable accuracy,
than shooting method and the regularized Benamou--Brenier approach.
\end{example}

\section{Conclusion}

We developed a structure-preserving Newton method for dynamical optimal transport on finite weighted graphs in the Benamou–Brenier framework, computing discrete Wasserstein geodesics via a fully coupled nonlinear system from finite-difference discretization of the graph continuity equation and Hamiltonian dynamics. 

Several extensions are suggested by the present study. 
On the discretization side, it would be valuable to develop higher-order time integrators that remain compatible with the graph continuity structure, preserve mass, and retain non-negativity (e.g., via SSP-type updates or positivity limiters), while clarifying how the Jacobian invertibility conditions extend beyond the first-order scheme used here. Another important direction is to extend the present framework beyond finite weighted undirected graphs to more general network settings, including directed and time-dependent graphs, as well as unbalanced optimal transport on graphs. More broadly, these developments could strengthen the role of the proposed solver as a building block for higher-level tasks, such as topology inference, graph Wasserstein barycenters, and gradient-flow-type problems.

\section{Acknowledgment}
We thank Daniyar Omarov for providing us with the code used to solve the Monge--Amp\`ere equation in Example~\ref{MAvsBB}. 
We also acknowledge the use of ChatGPT solely for double checking grammar and typographical errors in the paper.

\bibliographystyle{plain}
\bibliography{reference}

\vspace{9mm}

\appendix

\section{Invertibility of the Jacobian for the upwind scheme}
\label{sec:invertibility-upwind}

In this section, we verify the nonsingularity of the generalized
Jacobian associated with the fully discrete upwind scheme. We follow
the same reduction argument as in Section~\ref{Newton iteration} and
retain the notation introduced there. The only modification is that the
residual mapping $\tilde{\mathbf F}$ associated with the original
discretization is replaced by the residual mapping
$\tilde{\mathbf F}_{\rm up}$, obtained by using the upwind weight in the
discrete equations. All variables, index conventions, and block
structures are inherited from the main text unless otherwise stated.

For the edge $(i,j)$, define the upwind weight by
\[
\theta_{ij}^{\rm up}(\rho,v)
:=
\rho_i\,\mathbf 1_{\{v_{i,j}\ge 0\}}
+
\rho_j\,\mathbf 1_{\{v_{i,j}<0\}},
\]
where $\mathbf 1_A$ denotes the indicator function of the set $A$.
For each fixed velocity configuration, $\theta_{ij}^{\rm up}$ is affine
in the density variables. In particular,
\[
\frac{\partial \theta_{ij}^{\rm up}(\rho,v)}{\partial \rho_i}
=
\mathbf 1_{\{v_{i,j}\ge 0\}},
\qquad
\frac{\partial \theta_{ij}^{\rm up}(\rho,v)}{\partial \rho_j}
=
\mathbf 1_{\{v_{i,j}<0\}}.
\]

With this convention, the residual mapping $\tilde{\mathbf F}_{\rm up}$
has the same block decomposition as $\tilde{\mathbf F}$:
\begin{equation}
\label{the upwind F}
\tilde{\mathbf F}_{\rm up}
=
\begin{pmatrix}
(\tilde{\mathbf F}_{\rm up})_{\hat\rho}^{\rm int}\\
(\tilde{\mathbf F}_{\rm up})_{\hat\rho}^{\rm fin}\\
(\tilde{\mathbf F}_{\rm up})_{\hat v_T}
\end{pmatrix}.
\end{equation}
Here the density and velocity residuals are modified by replacing the
original weight function with the upwind weight, while the remaining
algebraic structure of the discretization is kept unchanged. Since the
upwind weight is piecewise smooth with respect to the velocity
variables, the mapping $\tilde{\mathbf F}_{\rm up}$ is, in general, not
classically differentiable with respect to $\hat v_T$ on the switching
set. Therefore, the Jacobian matrices below are understood as elements,
or blocks of elements, of the Clarke generalized Jacobian \cite{clarke1983optimization}.

We first record the generalized derivatives of the typical nonsmooth
terms appearing in the upwind residual. For the density equation, fix a
time level $t_m$. The velocity-dependent upwind flux is
\[
v_{i,j}^m\theta_{ij}^{\rm up}(\rho^m,v^m)
=
v_{i,j}^m
\left(
\rho_i^m\mathbf 1_{\{v_{i,j}^m\ge 0\}}
+
\rho_j^m\mathbf 1_{\{v_{i,j}^m<0\}}
\right)
=
\rho_i^m (v_{i,j}^m)^+
+
\rho_j^m (v_{i,j}^m)^-,
\]
where
\[
(v_{i,j}^m)^+ := \max\{v_{i,j}^m,0\},
\qquad
(v_{i,j}^m)^- := \min\{v_{i,j}^m,0\}.
\]
For fixed $\rho^m$, this function is locally Lipschitz with respect to
the scalar variable $v_{i,j}^m$. If $v_{i,j}^m\neq 0$, then
\[
\frac{\partial}{\partial v_{i,j}^m}
\left(
v_{i,j}^m\theta_{ij}^{\rm up}(\rho^m,v^m)
\right)
=
\begin{cases}
\rho_i^m, & v_{i,j}^m>0,\\
\rho_j^m, & v_{i,j}^m<0.
\end{cases}
\]
At the switching point $v_{i,j}^m=0$, the Clarke generalized derivative
is
\[
\partial_C^{v_{i,j}^m}
\left(
v_{i,j}^m\theta_{ij}^{\rm up}(\rho^m,v^m)
\right)
=
\operatorname{co}\{\rho_i^m,\rho_j^m\}.
\]
Since $\rho_i^m$ and $\rho_j^m$ are scalar quantities,
\[
\operatorname{co}\{\rho_i^m,\rho_j^m\}
=
[\min\{\rho_i^m,\rho_j^m\},\max\{\rho_i^m,\rho_j^m\}].
\]
Consequently,
\[
\partial_C^{v_{i,j}^m}
\left[
-\sqrt{\omega_{ij}}\,
v_{i,j}^m\theta_{ij}^{\rm up}(\rho^m,v^m)
\right]
=
\begin{cases}
\{-\sqrt{\omega_{ij}}\rho_i^m\}, & v_{i,j}^m>0,\\
\{-\sqrt{\omega_{ij}}\rho_j^m\}, & v_{i,j}^m<0,\\
-\sqrt{\omega_{ij}}\operatorname{co}\{\rho_i^m,\rho_j^m\},
& v_{i,j}^m=0.
\end{cases}
\]

For the velocity equation, we use the antisymmetry convention
$v_{k,j}=-v_{j,k}$. Then a typical upwind term satisfies
\[
(v_{k,j})^2\mathbf 1_{\{v_{j,k}\ge 0\}}
=
(v_{j,k}^+)^2,
\qquad
v_{j,k}^+ := \max\{v_{j,k},0\}.
\]
Thus, by the definition of the Clarke generalized derivative,
\[
\partial_C
\left[
(v_{k,j})^2\mathbf 1_{\{v_{j,k}\ge 0\}}
\right]
=
\begin{cases}
\{2v_{j,k}\}, & v_{j,k}>0,\\
\{0\}, & v_{j,k}\le 0.
\end{cases}
\]
Hence the non-smooth terms in the velocity residual have well-defined
Clarke generalized derivatives with respect to the edge velocity
variables.

Therefore, as in Theorem~\ref{main-thm}, the nonsingularity of the
full generalized Jacobian of $\tilde{\mathbf F}_{\rm up}$ reduces to
the nonsingularity of the two auxiliary blocks appearing in the
original proof. The corresponding entries are modified only through the
replacement of $\theta_{ij}$ by $\theta_{ij}^{\rm up}$ and of the
classical derivatives by the Clarke generalized derivatives described
above.

\begin{theorem}[Upwind scheme]\label{main-thm-upwind}
Let 
$\bigl((\bar{\rho}_i(t))_{i\le N-1}, \bar{v}_T(t)\bigr)$, $t\in[0,1],$
be the unique solution of~\eqref{equation expression1} such that
\[
\bigl((\bar{\rho}_i(t))_{i\le N-1}, \bar{v}_T(t)\bigr)\in 
C^{1}\bigl([0,1];\mathbb{R}^{2N-2}\bigr),
\qquad
\min_{t\in[0,1]}\min_{i=1,\dots,N}\bar{\rho}_i(t)\ge c>0.
\]
Let $t_1=0$, $t_m=m\tau$ for $m\le M+1$, and $t_{M+1}=1$.
$\tilde{\mathbf F}_{\rm up}$ is defined in \eqref{the upwind F}
and the discrete unknown vector is denoted by
$$\bar x:=
\Bigl(
(\bar{\rho}_i(t_m))_{1\le m\le M-1,\ i\le N-1},\,
(\bar v_T(t_m))_{0\le m\le M}
\Bigr)^\top.$$

Assume that $\tilde{\mathbf F}_{\rm up}:\mathbb{R}^{(M-1)(N-1)+(M+1)|E_T|}\to 
\mathbb{R}^{(M-1)(N-1)+(M+1)|E_T|}$ is \emph{locally Lipschitz} and \emph{semismooth}
in a neighborhood of $\bar{x}$.
Moreover, assume that the Clarke generalized Jacobian is uniformly nonsingular at $\bar x$, in the sense that
\begin{equation}\label{eq:clarke-nonsing}
\inf_{J\in \partial_C(\tilde{\mathbf F}_{\rm up})^{\rm fin}_{\hat\rho}(\bar x)}
\sigma_{\min}(\hat J_{\mathcal V}) \ge \gamma >0 .
\end{equation}
Equivalently, every matrix
$\hat J_{\mathcal V}\in
\partial_C(\tilde{\mathbf F}_{\rm up})^{\rm fin}_{\hat\rho}(x)
\Big|_{x=(G(\Gamma(\mathcal V),\mathcal V)^\top,\,
(\Gamma(\mathcal V),\mathcal V)^\top)}$
is invertible for $\mathcal V$ in a sufficiently small neighbourhood of $\bar{\mathcal V}$.

Let $x^{(0)}$ be an initial guess satisfying $\lvert x^{(0)}-\bar{x}\rvert=\mathcal{O}(\eta)$
for $\eta>0$ sufficiently small. Consider the semismooth Newton iteration:
choose any $J^{(k)}\in \partial_C \tilde{\mathbf F}_{\rm up}(x^{(k)})$ and compute $J^{(k)} d^{(k)} = -\tilde{\mathbf F}_{\rm up}(x^{(k)})$, $x^{(k+1)} = x^{(k)} + d^{(k)}.$
Then the iteration is well-defined for all $k$ sufficiently large and converges \emph{locally superlinearly} to $\bar{x}$.
If, in addition, $\tilde{\mathbf F}_{\rm up}$ is \emph{strongly semismooth} at $\bar{x}$,
then the convergence is \emph{quadratic}.
\end{theorem}

\begin{proof}
 The result follows from Theorem~\ref{main-thm} once we verify the
uniform nonsingularity of the generalized Jacobian associated with the
upwind discretization. 

Applying the upwind weight functions $\theta^{\mathrm{up}}_{ij}$ and left rectangular formula to \eqref{equation expression1}, we note that, for each fixed velocity configuration, the upwind weight
function $\theta^{\mathrm{up}}_{ij}$ is affine in the density variables
$\hat\rho$, and hence is classically differentiable with respect to
$\hat\rho$. Consequently, the $\hat\rho$-block of every element of the
Clarke generalized Jacobian of $\tilde{\mathbf F}_{\rm up}$ coincides
with the corresponding classical Jacobian block.
From the definitions of
\[
A := \frac{\partial (\tilde{\mathbf F}_{\rm up})_{\hat{\rho}}^{\rm int}}{\partial \hat{\rho}}
\in \mathbb{R}^{(M-1)(N-1)\times (M-1)(N-1)}, \hat{\rho} := (\rho^{m}_i)_{2\le m \le M,i\le N-1},
\]
and
\[
D_0 \in 
\partial^{\widetilde{\mathcal V}}_C
\left[
(\tilde{\mathbf F}_{\rm up})_{\hat v_T}
\right]
\Big|_{\big({\mathcal H}(\hat v_T)^\top,\hat v_T^\top\big)^\top}
\subset
\mathbb{R}^{(N-1)M\times (N-1)M},
\]
where $\partial_C^{\widetilde{\mathcal V}}$ denotes the Clarke
generalized Jacobian with respect to the variables
$\widetilde{\mathcal V}$.
It follows that \(A\) is lower triangular with identity blocks on the diagonal. Therefore, \(A\) is invertible. Based on the Clark derivatives, we derive the structure of the matrix $D_0$. 
For a fixed time index $m$, consider the residual corresponding to the
discretized velocity equation at this time step. The diagonal block of the
Jacobian is obtained by differentiating this residual with respect to the
velocity unknown, namely the unknown at time level
$t_{m+1}$ under our indexing convention. Thus, this derivative contributes to
the diagonal block of the Jacobian. Consequently, $D_0$ is lower
triangular with identity blocks on the diagonal, and thus $D_0$ is invertible as well.

Consequently, all the assumptions of Theorem~\ref{main-thm} are satisfied for the upwind discretization, and the conclusion of the theorem also holds for the upwind scheme. 
\end{proof}
\end{document}